\documentclass[11pt]{amsart}

\usepackage[a4paper,hmargin=3.5cm,vmargin=4cm]{geometry}
\usepackage{amsfonts,amssymb,amscd,amstext}
\usepackage{graphicx}
\usepackage[dvips]{epsfig}

\usepackage{fancyhdr}
\input xy
\xyoption{all}

\usepackage{times}

\usepackage{enumerate}
\usepackage{titlesec}
\usepackage{mathrsfs}

\titleformat{\section}%[display]
{\filcenter\bfseries\large} {\thesection{.}}{0.2cm}{}%[$\vspace*{-1.0cm}$]
\titleformat{\subsection}[runin]
{\bfseries} {\thesubsection{.}}{0.15cm}{}[.]
\titleformat{\subsubsection}[runin]
{\em}{\thesubsubsection{.}}{0.15cm}{}[.]
\usepackage[up,bf]{caption}
\newtheorem{theorem}{Theorem}[section]

\newtheorem{claim}[theorem]{Claim}
\newtheorem{lemma}[theorem]{Lemma}
\newtheorem{corollary}[theorem]{Corollary}

\theoremstyle{definition}

\newtheorem{remark}[theorem]{Remark}

\newtheorem{problem}[theorem]{Problem}

\numberwithin{equation}{section}
\numberwithin{figure}{section}

\def\Pcal{\mathcal{P}}

\def\Ocal{\mathcal{O}}
\def\Rcal{\mathcal{R}}

\def\be{\mathbf{e}}
\def\bu{\mathbf{u}}
\def\bv{\mathbf{v}}
\def\bx{\mathbf{x}}
\def\by{\mathbf{y}}
\def\bw{\mathbf{w}}
\def\bz{\mathbf{z}}

\def\Ascr{\mathscr{A}}
\def\Bscr{\mathscr{B}}
\def\Cscr{\mathscr{C}}

\def\Dscr{\mathscr{D}}
\def\Lscr{\mathscr{L}}

\def\c{\mathbb{C}}

\def\z{\mathbb{Z}}
\def\d{\mathbb{D}}
\def\b{\mathbb{B}}
\def\r{\mathbb{R}}
\def\n{\mathbb{N}}

\def\t{\mathbb{T}}
\def\z{\mathbb{Z}}

\def\igot{\mathfrak{i}}

\def\igot{\mathfrak{i}}

\def\pgot{\mathfrak{p}}

\renewcommand\imath{\igot}

\def\Agot{\mathfrak{A}}

\def\Ygot{\mathfrak{Y}}

\def\ttI{\mathtt{I}}
\def\ttJ{\mathtt{J}}

\def\dist{\mathrm{dist}}
\def\span{\mathrm{span}}
\def\length{\mathrm{length}}
\def\Flux{\mathrm{Flux}}

\def\cd{\overline{\mathbb D}}
\def\d{\mathbb D}

\newcommand\wt{\widetilde}
\newcommand\wh{\widehat}

\newcommand\di{\partial}
\newcommand\dibar{\overline\partial}
\newcommand\hra{\hookrightarrow}

\newcommand\CMI{\mathrm{CMI}}

\newcommand\E{\mathrm{e}}

\def\bx{\mathbf{x}}
\def\by{\mathbf{y}}

\usepackage{color}

\begin{document}

%\fancyhead[LO]{Riemann-Hilbert problem for null curves in $\c^n$}
%\fancyhead[LO]{Complete minimal surfaces bounded by Jordan curves}
%\fancyhead[LO]{Complete bounded minimal surfaces}
\fancyhead[LO]{Complete minimal surfaces bounded by Jordan curves}
\fancyhead[RE]{A.\ Alarc\'on, B.\ Drinovec Drnov\v sek, F.\ Forstneri\v c \& F.\ J.\ L\'opez}
\fancyhead[RO,LE]{\thepage}

\thispagestyle{empty}

%% Title
\vspace*{7mm}
\begin{center}
{\bf \LARGE Every bordered Riemann surface is a complete conformal minimal surface bounded by Jordan curves}
\vspace*{7mm}

%% Authors
{\large\bf A.\ Alarc\'on, B.\ Drinovec Drnov\v sek, F.\ Forstneri\v c \& F.\ J.\  L\'opez}
\end{center}

%% Addresses and finantial support
%\footnote[0]{\vspace*{-0.4cm}
%}
%% Abstract, keywords, and MSC

\vspace*{8mm}

\begin{quote}
{\small
\noindent {\bf Abstract}\hspace*{0.1cm}
In this paper we find approximate solutions of certain Riemann-Hilbert boundary value
problems for  minimal surfaces in $\r^n$ and null holomorphic curves in $\c^n$ for any $n\ge 3$.
With this tool in hand we construct complete conformally immersed 
minimal surfaces in $\r^n$ which are normalized by any given 
bordered Riemann surface and have Jordan boundaries. 
We also furnish complete conformal proper minimal immersions from any given 
bordered Riemann surface to any smoothly bounded, strictly convex domain 
of $\r^n$ which extend continuously up to the boundary; for $n\ge 5$ we find embeddings with these properties. 

\vspace*{0.1cm}

\noindent{\bf Keywords}\hspace*{0.1cm} Riemann surfaces, minimal surfaces, null curves.

\vspace*{0.1cm}

\noindent{\bf MSC (2010):}\hspace*{0.1cm} 53A10; 32B15, 32E30, 32H02.
}
\end{quote}

%%%%%%%%%%
%%%%%%%%%%
%%%%%%%%%%
%%%%%%%%%%
%%%%%%%%%%
%%%%%%%%%%

\section{Introduction} 
\label{sec:intro}  

In this paper we introduce a new tool -- the Riemann-Hilbert method -- into the study of minimal surfaces in the real 
Euclidean space $\r^n$, and null curves in the complex Euclidean space $\c^n$, for any $n\ge3$, 
and we obtain several applications. A special case of this technique is already available for $n=3$ (cf.\ \cite{AF2}), 
but the general case treated here, especially for $n>3$,  is more subtle and requires a novel approach.

Our first aim is to delve into the analysis of global geometric properties of minimal surfaces in $\r^n$ 
bounded by Jordan curves. The classical Plateau problem amounts to finding a minimal surface spanning 
a given contour. In 1931, Douglas \cite{Do} and Rad\'o \cite{Ra} independently solved this problem  
for any Jordan curve in $\r^n$.  A major topic in global theory is the study of geometry of {\em complete minimal surfaces}, 
that is, minimal surfaces which are complete in the intrinsic distance.  
The analysis of the asymptotic behavior, the conformal structure, and the influence of topological embeddedness are
central questions in this field; see  \cite{MP1,MP2} for recent surveys.
By the isoperimetric inequality, minimal surfaces in $\r^n$ spanning rectifiable Jordan curves are not complete.
Our first main result provides complete minimal surfaces with (nonrectifiable) Jordan boundaries 
in $\r^n$ for $n\ge 3$ which are normalized by any given bordered Riemann surface.

%
%
%   THEOREM 1.1
%
%
\begin{theorem}\label{th:Jordan}
Let $M$ be a compact bordered Riemann surface. 
Every conformal minimal immersion $F\colon M\to\r^n$ $(n\geq 3)$ of class $\Cscr^1(M)$ 
can be approximated arbitrarily closely in the $\Cscr^0(M)$ topology by a continuous map  $\wt F\colon M\to\r^n$ 
such that $\wt F|_{M\setminus bM}\colon M\setminus bM\to\r^n$ 
is a conformal complete minimal immersion,  $\wt F|_{bM}\colon bM\to\r^n$ is a topological embedding, and the  flux
of $\wt F$ equals the one of $F$. In particular, $\wt F(bM)$ consists of finitely many Jordan curves.
If $n\ge 5$ there exist embeddings $\wt F \colon M\hra\r^n$ with these properties. 
\end{theorem}

Theorem \ref{th:Jordan} shows that every finite collection of smooth Jordan curves in $\r^n$ spanning a connected minimal 
surface can be approximated in the $\Cscr^0$ topology by families of Jordan curves spanning {\em complete} connected minimal 
surfaces; hence it can be viewed as an {\em approximate solution of the Plateau problem by complete minimal surfaces.}

Recall that a {\em compact bordered Riemann surface} is a compact connected surface $M$, endowed with a 
complex (equivalently, a conformal) structure, whose boundary $bM\neq\emptyset$ consists of finitely many smooth Jordan curves.  
The interior $\mathring M=M\setminus bM$ of such $M$ is an (open) {\em bordered Riemann surface}. 
A {\em conformal minimal immersion} $F\colon M\to\r^n$ is an immersion 
which is angle preserving and harmonic; every such map parametrizes a minimal surface in $\r^n$
(see Sec.\ \ref{sec:prelim}). The {\em flux} of $F$ 
is the group homomorphism $\Flux (F)\colon H_1(M,\z)\to \r^n$ given on any closed curve 
$\gamma\subset M$ by $\Flux(F)(\gamma)=\int_\gamma \Im (\partial F)$, 
where $\partial F$ is the complex differential of $F$ (see (\ref{eq:flux}) below) 
and $\Im$ denotes the imaginary part.

Theorem \ref{th:Jordan} generalizes pioneering results of Mart\'{i}n and Nadirashvili \cite{MN} who dealt with 
immersed minimal discs in $\r^3$. Their method relies on a refinement of Nadirashvili's construction of a complete 
bounded minimal disc in $\r^3$ \cite{Na} and is based on a recursive application of classical Runge's approximation 
theorem. By using the same technique, Alarc\'on \cite{A1} constructed compact complete minimal immersions in 
$\r^3$ with arbitrary finite topology; i.e., continuous maps 
$F\colon \overline \Omega\to\r^3$ such that $F|_\Omega\colon\Omega\to\r^3$ is a conformal complete minimal immersion, 
where $\Omega$ is a relatively compact domain in an open Riemann surface. However, neither the conformal structure 
of $\Omega$, nor the topology of its boundary, can be controlled by this method; in particular, it can not be ensured that 
$F(\overline\Omega\setminus\Omega)\subset \r^3$ consists of Jordan curves. 
Indeed, the use of Runge's theorem in Nadirashvili's  technique does not enable one to control the placement in $\r^3$ 
of the entire surface, and hence one must cut away pieces of the surface in order to keep it suitably bounded; this surgery
causes the aforementioned problems. By a different technique, relying on Runge-Mergelyan type theorems 
(cf.\ \cite{AL1,AL-No}),  Alarc\'on and L\'opez obtained analogous results for nonorientable minimal surfaces in $\r^3$ \cite{AL-Ritore}, 
null holomorphic curves in $\c^3$, and complex curves in $\c^2$ \cite{AL-Israel}. (Recall that a null curve in $\c^n$, $n\geq 3$, 
is a complex curve whose real and imaginary parts are minimal surfaces in $\r^n$.) 
Their technique still does not suffice to control the conformal structure of the surface or the topology of its boundary.

By introducing the Riemann-Hilbert method into the picture,  Alarc\'on and Forstneri\v c  \cite{AF2} 
recently constructed complete bounded minimal surfaces in $\r^3$, and null curves in $\c^3$, normalized by 
any given bordered Riemann surface.   The principal advantage of the Riemann-Hilbert method over Runge's theorem
in this problem is that it enables one to work on a fixed bordered Riemann surface, controlling its global 
placement in $\r^n$ or $\c^n$ at all stages of the construction.

The main novelty of Theorem \ref{th:Jordan} is that {\em we prescribe both the complex structure} 
(any bordered Riemann surface) {\em and the asymptotic behavior} 
(bounded by Jordan curves) of complete bounded minimal surfaces; furthermore, we obtain 
results in any dimension $n\ge 3$. This is achieved by developing the Riemann-Hilbert technique,
first used in \cite{AF2} for $n=3$,  in any dimension $n\ge 3$, and by further improving its implementation
in the recursive process. Theorem \ref{th:Jordan} is new even in the case $n=3$. 
Furthermore, for $n>3$ this seems to be the first known approximation result by complete bounded minimal surfaces, 
even if one does not take care of the conformal structure of the surface and the asymptotic behavior of its boundary. 
Previous results in this line are known for complex curves in $\c^n$, $n\geq 2$, 
and null curves in $\c^n$, $n\geq 3$; cf.\ \cite{AF,AF2,AL-CY}.

The Riemann-Hilbert method developed in this paper also allows us to establish essentially optimal results 
concerning {\em proper complete minimal surfaces in convex domains}; see Theorems \ref{th:proper} and \ref{th:topology} below.
These results, and the methods used in their proof, will hopefully provide a step towards the more ambitious goal 
of finding optimal geometric conditions on a domain $D\subset\r^n$ for $n\ge 3$ which guarantee that $D$ admits plenty 
of proper (possibly also complete) conformal minimal immersions from any given bordered Riemann surface. 
An explicit motivation comes from the paper \cite{DF2007} on proper holomorphic images
of bordered Riemann surfaces in complex manifolds endowed with an exhaustion function whose Levi form 
has at least two positive eigenvalues at every point. % near infinity.

We shall say that a domain $\Dscr \subset\r^n$ is {\em smoothly bounded} if it is 
bounded and  its boundary $b\Dscr=\overline \Dscr\setminus \Dscr$ is smooth (at  least of class $\Cscr^2$).

%
%
%   THEOREM 1.2
%
%
\begin{theorem}\label{th:proper}
Let $\Dscr\subset\r^n$  $(n\ge 3)$ be a bounded strictly convex domain with $\Cscr^2$ smooth boundary,
let $M$ be a compact bordered Riemann surface, and let $F\colon M\to \overline\Dscr$ be a conformal minimal immersion 
of class $\Cscr^1(M)$. Then the following assertions hold:
\begin{enumerate}[\rm (a)]
\item If $F(M)\subset \Dscr$ then $F$ can be approximated uniformly on compacts in $\mathring M = M\setminus bM$ 
by continuous maps $\wt F\colon M\to \overline \Dscr$ such that   
$\wt F|_{\mathring M}\colon \mathring M \to\Dscr$ is a conformal complete proper minimal immersion with 
$\Flux(\wt F)=\Flux(F)$. 
\vspace{1mm}
\item If $F(bM)\subset b\Dscr$ then $F$ can be approximated in the $\Cscr^0(M)$ topology
by continuous maps $\wt F\colon M\to\overline \Dscr$ such that  
$\wt F|_{\mathring M}\colon \mathring M \to\Dscr$ is a conformal complete 
proper minimal immersion.
\end{enumerate}
In either case, the frontier $\wt F(bM)\subset b\Dscr$ consists of finitely many closed curves.  
If $n\geq 5$ then the approximation can be achieved by maps $\wt F$ which are embeddings on $\mathring M$.
\end{theorem}

Theorem \ref{th:proper} is proved in Sec.\ \ref{sec:topology}. 
It may be viewed as a version of Theorem \ref{th:Jordan} in which we additionally ensure that 
the boundary curves of a minimal surface are contained in the boundary of the domain, at the cost of losing
topological embeddedness of these curves. A partial result in this direction 
can be found in \cite{A2} where the first named author constructed compact complete proper minimal immersions of 
surfaces with arbitrary finite topology into smoothly bounded strictly convex domains in $\r^3$,
but without control of the conformal structure on the surface or the flux of the immersion. 

In part {\rm (a)} of Theorem \ref{th:proper} it is clearly impossible to ensure approximation in the $\Cscr^0(M)$ topology; 
however, a nontrivial upper bound for the maximum norm $\|\wt F-F\|_{0,M}$, 
depending on the placement of the boundary $F(bM)\subset\Dscr$, is provided by Theorem \ref{th:proper2}. 
In part (b) the flux can be changed by an arbitrarily small amount.  
Unlike in Theorem \ref{th:Jordan}, we are unable to guarantee that the boundary $\wt F(bM)\subset b\Dscr$ 
consists of Jordan curves; see Remark \ref{rem:Jordan}. Hence the following remains an open problem. 

\begin{problem}
Let $\Dscr$ be a smoothly bounded, strictly convex domain  in $\r^n$ for some $n\ge 3$. 
Does there exist a complete proper minimal surface in $\Dscr$ bounded by finitely
many Jordan curves in the boundary $b\Dscr$ of $\Dscr$? What is the answer if $\Dscr$ is the unit ball of $\r^n$?
\end{problem}

Theorem \ref{th:proper} fails in general for weakly convex domains.
For instance, no polyhedral region of $\r^3$ admits a complete proper minimal disc that is
continuous up to the boundary \cite{AN,Na2}; it is easily seen that such a disc 
would violate Bourgain's theorem on the radial variation of bounded analytic functions \cite{Bourgain}.
(See also \cite{G2} for the case of complex discs in a  bidisc of $\c^2$.)
Our next result shows that the situation is rather different if we do not insist on continuity up to the boundary. 

%
%
%   WEAKLY CONVEX DOMAINS
%
%
\begin{theorem}\label{th:topology}
Let $D$ be a convex domain in $\r^n$ for  some $n\ge 3$. 
\begin{enumerate}[\rm (a)]
\item If $M$ is a compact bordered Riemann surface and $F\colon M\to D$ is a conformal minimal immersion 
of class $\Cscr^1(M)$, then $F$ can be approximated uniformly on compacts in $\mathring M$ 
by conformal complete proper minimal immersions $\wt F\colon \mathring M\to D$ with $\Flux(\wt F)=\Flux(F)$. 
If $n\geq 5$ then $\wt F$ can be chosen an embedding.
\vspace{1mm}
\item Every open orientable smooth surface $S$ carries a full complete proper minimal immersion 
$S\to D$ (embedding if $n\ge 5$) with arbitrary flux.  
\end{enumerate}
\end{theorem}

Recall that a minimal surface in $\r^n$ is said to be {\em full} if it is not contained in any affine hyperplane.
In part (b), the flux is meant with respect to the conformal structure induced on the surface
$S$ by the Euclidean metric of $\r^n$ via the immersion $S\to\r^n$. 

Theorem \ref{th:topology} is proved in Sec.\ \ref{sec:topology}. The case $D=\r^n$ was already established in \cite{AFL,AL1} 
where conformal complete proper minimal immersions $\mathcal{R}\to\r^n$  (embeddings if $n\geq 5$) with arbitrary flux are constructed 
for every open Riemann surface $\mathcal{R}$. However, since the existence of a nonconstant positive harmonic function 
is a nontrivial condition on an open Riemann surface, it is clearly impossible to prescribe the conformal type of a 
full minimal surface in any either convex or smoothly bounded domain $D$ different from $\r^n$.  
Ferrer, Mart\'in, and Meeks \cite{FMM} proved Theorem \ref{th:topology} {\rm (b)} for $n=3$ 
but without the control of the flux, whereas the cases $n\in\{3,4,6\}$ and vanishing flux follow from the results in \cite{AF1,AL-CY}. 

If one is merely interested  in the existence of proper minimal surfaces (without approximation),
then it suffices to assume that the domain $D\subset \r^n$ admits a smooth strongly convex boundary point $p\in bD$. 
Indeed, by using the approximation statement in Theorem \ref{th:proper2} one can find proper conformal 
minimal immersions into $D$ with boundaries in a small neighbourhood of $p$ in $bD$, 
thereby proving the following corollary (see Sec.\ \ref{sec:topology}).

%
%
%  COROLLARY ON DOMAINS WITH A STRONGLY CONVEX BOUNDARY POINT
%
\begin{corollary}\label{co:bdddomains}
If $D\subset\r^n$ $(n\ge 3)$ is a domain with a $\Cscr^2$ smooth strictly convex boundary point then the following hold.
\begin{enumerate}[\rm (a)]
\item 
Every compact bordered Riemann surface $M$ admits a continuous map $\wt F\colon M \to \overline D$ 
such that $\wt F(\mathring M)\subset D$ and  $\wt F|_{\mathring M}\colon \mathring M \to D$ is a conformal 
full complete proper minimal immersion. If $n\geq 5$ then $\wt F$ can be taken to be an embedding on $\mathring M$.
\vspace{1mm}
\item  
Every open orientable smooth surface $S$ carries a full complete proper minimal immersion $S\to D$ 
(embedding if $n\ge 5$) with arbitrary flux.
\end{enumerate}
\end{corollary}

Note that every smoothly bounded relatively compact domain in $\r^n$ admits a strictly convex boundary point,
so Corollary \ref{co:bdddomains} applies to such domains.

To the best of our knowledge, Theorems \ref{th:Jordan}, \ref{th:proper}, and \ref{th:topology}, 
and Corollary \ref{co:bdddomains} provide the first examples of complete bounded {\em embedded} 
minimal surfaces in $\r^5$ with controlled topology; furthermore, they solve the {\em conformal  Calabi-Yau problem} 
for {\em embedded bordered Riemann surfaces in convex domains of $\r^n$} for any $n\geq 5$. 

Recall that the {\em Calabi-Yau problem} deals with the existence and geometric properties of 
complete bounded minimal surfaces; see for instance \cite{AL-CY,AF2} and the references therein for
the state of the art of this subject. Regarding the {\em embedded Calabi-Yau problem},  it is still unknown whether there 
exist complete bounded embedded minimal surfaces in $\r^3$. By results of Colding and Minicozzi \cite{CM} and 
Meeks, P\'erez, and Ros \cite{MPR}, there is no such surface with finite genus and at most countably many ends. 
Recently Alarc\'on and L\'opez \cite{AL-JEMS} and Globevnik \cite{Gl,G3} constructed complete bounded 
embedded complex curves  (hence minimal surfaces) in $\c^2\equiv\r^4$;  
however, their method does not provide any information on the topology and the conformal structure of their examples. 
Globevnik  actually constructed a holomorphic function $f$ on the unit ball of $\c^n$ for any $n\ge 2$ \cite{Gl}
and, more generally, on any pseudoconvex domain  $D\subset \c^n$ for $n\ge 2$ \cite{G3}, 
such that every divergent curve in $D$ on which $f$ is bounded has infinite length. 
It follows that every level set $\{f=c\}$ of such a
function is a complete closed complex hypersurface in $D$.  

By applying the Riemann-Hilbert technique, developed in Sec.\ \ref{sec:RH} below, it is straightforward to
extend all main results of the paper \cite{DF2014}  to null hulls of 
compact sets in $\c^n$ and  minimal hulls of compact sets in $\r^n$ for any $n> 3$.
As pointed out in   \cite{DF2014}, the only reason for restricting to $n=3$ was that a 
Riemann-Hilbert lemma for null curves in dimension $n>3$ was not available at that time.
We postpone this to a subsequent publication.

The proofs of our results depend in an essential way on a new tool that we obtain in this paper, 
namely the Riemann-Hilbert boundary value problem for null curves in $\c^n$ and minimal surfaces in $\r^n$ 
for any $n\ge 3$ (cf.\ Theorems \ref{th:RH} and \ref{th:RHCMI}), generalizing the one developed  in 
\cite{AF2} in dimension $n=3$.  We also use  a number of  other recent ideas and  techniques: 
gluing holomorphic sprays on Cartan pairs (cf.\ \cite{DF2007} and \cite{F2011}),
the Mergelyan approximation theorem for conformal minimal immersions in $\r^n$ (cf.\ \cite{AL1,AFL}), 
the general position theorem for minimal surfaces in $\r^n$ for $n\ge 5$ (cf.\ \cite{AFL}), the method of exposing boundary points 
on a bordered Riemann surface (Forstneri\v c and Wold \cite{FW0}), and the circle of ideas around the construction of 
compact complete minimal immersions in $\r^3$ (cf.\ \cite{AL-Israel} and Mart\'in and Nadirashvili \cite{MN}) and complete 
bounded minimal surfaces in $\r^3$ normalized by  bordered Riemann surfaces (cf.\ \cite{AF2} and also \cite{AF}).

%%%%%%%%%%%%%%%%%%%%%%%%%%%%%
%
%     FURTHER APPLICATIONS
%
%%%%%%%%%%%%%%%%%%%%%%%%%%%%%

Our methods  easily adapt to give results analogous  to Theorems \ref{th:Jordan}, \ref{th:proper}, and \ref{th:topology} 
in the context of complex curves in $\c^n$, $n\geq 2$, and holomorphic null curves in $\c^n$, $n\geq 3$.
Indeed, all tools used in the proof  are available for these families of curves:  
the Riemann-Hilbert method for holomorphic null curves in arbitrary dimension is provided in this paper
(see Theorem \ref{th:RH}), while Runge-Mergelyan type theorems for null curves are proved in \cite{AF1,AL1}. 
For example, by following the proof of Theorem \ref{th:Jordan} one can show the following result.

\begin{theorem}
\label{th:further}
Every bordered Riemann surface $M$ admits a continuous map $F\colon M\to \c^n$, $n\ge 2$, such that 
$F|_{\mathring M}\colon \mathring M  \to\c^n$ is a complete holomorphic immersion
and $F(bM)$ is a finite union of Jordan curves. If $n\ge 3$ then there is a topological 
embedding $F\colon M\hra \c^n$ with these properties.
\end{theorem}

A similar result can be established for null curves $F\colon M\to\c^3$, $n\ge 3$; recall that the general position of 
null curves in $\c^n$ is embedded  for any $n\ge 3$ (cf.\ \cite[Theorem 2.4]{AF1}).

%
%
%   OUTLINE OF THE PAPER
%
%
\subsection*{Outline of the paper} 
In Sec.\ \ref{sec:prelim} we introduce the notation and preliminaries. 
In Sec.\ \ref{sec:RH} we develop the Riemann-Hilbert method for null curves 
(Lemmas \ref{lem:RH3}, \ref{lem:RH} and Theorem \ref{th:RH})  
and minimal surfaces  (Theorem \ref{th:RHCMI}) in arbitrary dimension $n\ge 3$. 
Sec.\ \ref{sec:Jordan} is devoted to the proof of Theorem \ref{th:Jordan}.
The main technical part is contained in Lemma \ref{lem:Jordan2}, asserting that any conformal minimal immersion 
$M\to\r^n$ of a compact bordered Riemann surface can be approximated in the $\Cscr^0(M)$ topology 
by conformal minimal immersions $M\to\r^n$ whose boundary distance from a fixed interior point is as big as desired. 
In Sec.\ \ref{sec:Jordan} we also prove that the general position of the boundary curves $bM\to\r^n$ 
of a conformal minimal immersion $M\to\r^n$ is embedded for $n\geq 3$; see Theorem \ref{th:gp}. 
Lemma \ref{lem:Jordan2} is also exploited in the proof of Theorem \ref{th:proper2} in Sec.\ \ref{sec:proper}.
The latter result is the key to the proof of Theorems \ref{th:proper}, \ref{th:topology} 
and Corollary \ref{co:bdddomains}  given  in Sec.\ \ref{sec:topology}.

%%%%%%%%%%
%%%%%%%%%%
%%%%%%%%%%
%%%%%%%%%%   NOTATION AND PRELIMINARIES
%%%%%%%%%%
%%%%%%%%%%

\section{Notation and preliminaries}\label{sec:prelim}

 We denote by $\langle\cdot,\cdot\rangle$, $\|\cdot\|$, and $\dist(\cdot,\cdot)$ the Euclidean scalar product, norm, and distance in $\r^n$, $n\in\n$.  Given a vector $\bx\in\r^n\setminus\{0\}$ we denote by 
 $\langle\bx\rangle^\bot=\{\bw\in\r^n\colon \langle\bw,\bx\rangle=0\}$ its orthogonal complement.
If $K$ is a compact topological space and $f\colon K\to\r^n$ is a continuous function, 
we denote by $\|f\|_{0,K}$ the maximum norm of $f$.

Set $\d=\{\zeta\in\c\colon |\zeta|<1\}$ and $\t= b{\d}=\{\zeta\in\c\colon |\zeta|=1\}$.

As usual we will identify $\c^n\equiv \r^{2n}$. We shall write $\imath=\sqrt{-1}$ . By $\Re(\bz)$ and $\Im(\bz)$ 
we denote the real and the imaginary part of a point $\bz\in\c^n$. Let $\bz=(z_1,\ldots,z_n)$ be complex
coordinates on $\c^n$. Denote by $\Theta$ the holomorphic bilinear form on $\c^n$ given by
\begin{equation}\label{eq:bilinear}
	\Theta(\bz,\bw)= \sum_{j=1}^n z_j w_j.
\end{equation}
Let $\Agot=\Agot^{n-1}\subset \c^n$ denote the {\em null quadric}  
\begin{equation}\label{eq:Agot}
	\Agot^{n-1} =\{\bz=(z_1,\ldots,z_n) \in\c^n\colon \Theta(\bz,\bz)=z_1^2+\ldots+z_n^2=0\}.
\end{equation}
This is a conical algebraic subvariety of $\c^n$ that is not contained in  any hyperplane of $\c^n$
and is nonsingular except at the origin. We also write $\Agot^{n-1}_*=\Agot^{n-1}\setminus\{0\}$.
In the sequel we shall omit the superscript when the dimension is clear from the context.

Let us recall the basic facts concerning minimal surfaces (see e.g.\ Osserman \cite{Osserman}).
Let $M$ be an open Riemann surface, and let $\theta$ a nowhere vanishing holomorphic $1$-form on $M$
 (such exists by the Oka-Grauert principle, cf.\ Theorem 5.3.1 in \cite[p.\ 190]{F2011}). 
The exterior derivative on $M$ splits into the sum $d=\di+\dibar$ of the $(1,0)$-part $\di$ and the $(0,1)$-part $\dibar$. 
An immersion $F=(F_1,\ldots,F_n)\colon M\to\r^n$ ($n\geq 3$) is {\em conformal} (angle preserving) 
if and only if its Hopf differential $\sum_{j=1}^n (\partial F_j)^2$ vanishes everywhere on $M$, 
that is to say, if $\partial F/\theta\in\Agot$ (\ref{eq:Agot}). A conformal immersion $F\colon M\to\r^n$ is 
minimal if and only if it is harmonic, and in this case $\Phi:=\partial F$ is a $\c^n$-valued  holomorphic $1$-form 
vanishing nowhere on $M$. Given a base point $p_0\in M$, $F$ can be written in the form
\begin{equation}\label{eq:FPhi}
	F(p)=F(p_0)+\Re\Big(\int_{p_0}^p \Phi\Big), \quad p\in M.
\end{equation}
This is called the {\em Weierstrass representation} of $F$.
Conversely, if an $n$-dimensional holomorphic 1-form $\Phi$ on $M$ has vanishing real periods 
(i.e.,  its real part $\Re(\Phi)$ is exact) and satisfies $\Phi/\theta\colon M\to \Agot_*$, then the map 
$F\colon M\to\r^n$ given by \eqref{eq:FPhi} is a conformal minimal immersion.

Let $H_1(M;\z)$ denote the first homology group of $M$ with integer coefficients. 
The {\em flux map} of a conformal minimal immersion $F\colon M \to\r^n$ is the group homomorphism 
$\Flux(F)\colon H_1(M;\z)\to\r^n$ given by
\begin{equation} \label{eq:flux}
	\Flux(F)(\gamma)=\Im\Big(\int_\gamma \partial F\Big) \quad \text{for every closed curve }\gamma\subset M.
\end{equation}
Since the $1$-form $\di F$ on $M$ is holomorphic and therefore closed, the integral on the right hand side is 
independent of the choice of a path in a given homology class.

Next we introduce the mapping spaces that will be used in the paper. 

If $M$ is an open Riemann surface then $\CMI(M,\r^n)$ denotes the set of all conformal minimal immersions $M\to \r^n$.

Assume now that $M$ is a {\em compact bordered Riemann surface}, i.e., a  compact connected Riemann surface with smooth boundary 
$\emptyset \ne bM \subset M$ and interior $\mathring M=M\setminus bM$. For any $r\in \z_+$  we denote by $\Ascr^r(M)$ the space 
of all functions $M\to \c$ of class $\Cscr^r(M)$ that are holomorphic in $\mathring M$. We write $\Ascr^0(M)=\Ascr(M)$. 
It is classical that every compact bordered Riemann surface $M$ can be considered as a smoothly bounded compact domain in 
an open Riemann surface $\wt M$ and, by Mergelyan's theorem, each function in $\Ascr^r(M)$ can be approximated in the
$\Cscr^r(M)$ topology by holomorphic functions on a neighborhood of $M$.  

For any $r\in \n$ we denote by $\CMI^r(M,\r^n)$ the set of all  conformal minimal immersions $M\to\r^n$ 
of class $\Cscr^r(M)$. More precisely, an immersion $F\colon M\to \r^n$ of class $\Cscr^r$ belongs to 
$\CMI^r(M,\r^n)$ if and only if $\di F$ is a $(1,0)$-form of class $\Cscr^{r-1}(M)$ 
which is holomorphic in the interior $\mathring M=M\setminus \partial M$ and has
range in the punctured null quadric $\Agot_*$ (\ref{eq:Agot}). 
For $r=0$ we define $\CMI^0(M,\r^n)$ as the class of all continuous maps $F\colon M\to \r^n$ such that
$F\colon\mathring M\to\r^n$ is a conformal minimal immersion.

By the local Mergelyan theorem for conformal minimal immersions \cite[Theorem 3.1 (a)]{AFL}, 
every $F\in \CMI^r(M,\r^n)$ for $r\ge 1$ can be approximated in the $\Cscr^r(M)$ topology by 
conformal minimal immersions on an open neighborhood of $M$ in the ambient surface $\wt M$. 
If $M$ is Runge in $\wt M$ then every such $F$ can also be approximated in the $\Cscr^r(M)$ topology 
by conformal minimal immersions $\wt M\to\r^n$ (cf.\ \cite[Theorem 5.3]{AFL}). 

We say that a holomorphic map $f\colon M\to \Agot_*$ is {\em nondegenerate}  if the image $f(M)\subset\Agot_*$ 
is not contained in  any complex hyperplane of $\c^n$. Clearly nondegenerate implies {\em nonflat},
where the latter condition means that $f(M)$ is not contained in a (complex) ray of the null cone $\Agot$.
If $f$ is nonflat then the linear span of the tangent spaces $T_{f(p)}\Agot$ over all 
points $p\in M$ equals $\c^n$ (cf.\ \cite[Lemma 2.3]{AFL}). 
The latter condition implies the existence of a dominating
and period dominating holomorphic spray of maps $f_w\colon M\to \Agot_*$, with the parameter
$w$ in a ball in some $\c^N$ and with the core map $f_0=f$ (see \cite[Lemma 5.1]{AF1} 
or \cite[Lemma 3.2]{AFL}). This will be used in the proof of Theorems \ref{th:RH} and \ref{th:RHCMI}.

A conformal minimal immersion $F\colon M\to \r^n$ is said to be {\em nondegenerate} if the map 
$f=\partial F/\theta \colon M\to \Agot_*$ is nondegenerate, and is said to be {\em full} if $F(M)$ 
is not contained in a hyperplane of $\r^n$. Nondegenerate conformal minimal immersions $M\to\r^n$ are full, 
but the converse is true only in the case $n=3$ (see \cite{Osserman}).

If $M$ is an open Riemann surface, we denote by $\CMI_*(M,\r^n)\subset \mathrm{CMI}(M,\r^n)$
the subset consisting of all immersions which are nondegenerate on every connected component of $M$.
The analogous notation is used for compact bordered Riemann surfaces:
$\CMI^r_*(M,\r^n)$ denotes the space of all  $F\in  \CMI^r(M,\r^n)$ which are nondegenerate 
on every connected component of $M$. By Theorem 3.1 in \cite{AFL},  $\CMI_*^r(M,\r^n)$ is a dense subset of 
$\CMI^1(M,\r^n)$ in the $\Cscr^1(M)$ topology for every $r\in\n$.

If $M$ is an open Riemann surface and $F\in\CMI(M,\r^n)$, we denote by $\dist_F(\cdot,\cdot)$ 
the intrinsic distance in $M$ induced by the Euclidean metric of $\r^n$ via $F$; i.e. 
\[
\dist_F(p,q)=\inf\{\length\, F(\gamma) \colon \gamma\subset M \text{ arc connecting $p$ and $q$}\},
\]
where $\length$ denotes the Euclidean length  in $\r^n$.
Likewise we define $\dist_F$ on $M$ when $M$ is a compact bordered Riemann surface. 
If $M$ is open, the immersion $F\colon M\to\r^n$ is said to be {\em complete} if $\dist_F$ 
is a complete metric on $M$;
equivalently, if the image $F(\gamma)$ of any divergent curve $\gamma\subset M$ 
(i.e., a curve which eventually leaves any compact subset of $M$) is a curve of 
infinite length in $\r^n$.

%%%%%%%%%%
%%%%%%%%%%
%%%%%%%%%%
%%%%%%%%%% RIEMANN-HILBERT METHOD FOR NULL CURVES AND MINIMAL SURFACES
%%%%%%%%%%
%%%%%%%%%%

\section{Riemann-Hilbert problem for null curves in $\c^n$}\label{sec:RH}

In this section we find approximate solutions of a general Riemann-Hilbert boundary value problem for null curves
and for confomal minimal immersions.

Let $n\in\n$, $n\geq 3$. Recall that $\Agot^{n-1}\subset \c^n$ is the null quadric
(\ref{eq:Agot}) and $\Agot^{n-1}_*=\Agot^{n-1}\setminus\{0\}$. 
We shall drop the superscript when the dimension $n$ is clear from the context.

We begin with the following essentially optimal result in dimension $n=3$ in which there is no restriction 
on the type of null discs attached at boundary points of the central null disc.
Lemma \ref{lem:RH3} generalizes \cite[Lemma 3.1]{AF2} which pertains to the case of
linear null discs of the form $\xi\mapsto r(\zeta)\,\xi \bu$ in a constant null direction 
$\bu\in \Agot_*$. The corresponding result for ordinary holomorphic discs can be found in several
sources, see e.g.\ \cite[Lemma 3.1]{DF2012}.

%
%
%  
%
%  RH PROBLEM FOR NULL DISCS IN C^n
%
%
%
%
\begin{lemma}  \label{lem:RH3}
Let $F\colon\overline{\d}\to\c^3$ be a null holomorphic disc of class $\Ascr^1(\d)$.
Assume that $I$ is a proper closed segment in the circle $\t=b\d$, 
$r\colon \t \to \r_+ := [0,1]$ is a continuous function supported on $I$ (the {\em size function}),  
and $\sigma \colon I \times\overline{\d}\to\c^3$ is a map of class $\Cscr^1$ such that
for every $\zeta\in I$ the map $\cd \ni \xi \mapsto \sigma(\zeta,\xi)$ is an immersed holomorphic null disc
with $\sigma(\zeta,0)=0$. Let $\varkappa\colon \t \times\overline{\d}\to\c^3$ be given by
\begin{equation}\label{eq:varkappa}
	\varkappa(\zeta,\xi)=F(\zeta) + \sigma\bigl(\zeta,r(\zeta)\, \xi\bigr)
\end{equation}
where we take $\sigma\bigl(\zeta,r(\zeta)\, \xi\bigr)=0$ for $\zeta\in\t\setminus I$.
Given numbers $\epsilon>0$, $0<\rho_0<1$ and an open neighborhood $U$ of $I$ in $\overline{\d}$, 
there exist a number $\rho'\in [\rho_0,1)$ and a null holomorphic immersion 
$G\colon\overline{\d}\to\c^n$ such that $G(0)=F(0)$ and the following conditions hold:
\begin{enumerate}[\it i)]
\item $\dist(G(\zeta),\varkappa(\zeta,\t))<\epsilon$ for all $\zeta\in \t$, 
\item $\dist(G(\rho\zeta),\varkappa(\zeta,\overline{\d}))<\epsilon$ for all $\zeta\in \t$ and all 
$\rho\in [\rho',1)$, and 
\item $G$ is $\epsilon$-close to $F$ in the $\Cscr^1$ topology on 
$(\overline{\d}\setminus  U) \cup \rho'\overline{\d}$.
\end{enumerate}
Moreover, given an upper semicontinuous function $\phi \colon\c^3\to \r\cup\{-\infty\}$, 
we may achieve in addition  that 
\begin{equation}
\label{eq:small-increase}
	\int_{I} \phi \bigl( G(\E^{\imath t})\bigr) \, \frac{dt}{2\pi} \le 
 	\int^{2\pi}_0 \!\! \int_{I} \phi \bigl(\varkappa(\E^{\imath t},\E^{\imath s})\bigr) 
       \frac{dt}{2\pi} \frac{ds}{2\pi} + \epsilon
\end{equation}
\end{lemma} 
% Lemma RH3

% Condition (\ref{eq:small-increase}) will not be used in this paper.

\begin{proof}
Let $\pi\colon\c^2\to\c^3$ be the homogeneous quadratic map defined by
\begin{equation}\label{eq:pi}
	\pi(u,v)=\left(u^2-v^2,2uv,-\imath(u^2+v^2)\right), \quad (u,v)\in\c^2. 
\end{equation}
Note that $\pi$ is a two-sheeted parametrization of the null quadric
$\Agot\subset \c^3$ (\ref{eq:Agot}), commonly called the 
{\em spinor parametrization} of $\Agot$, and $\pi$ is branched only at the point $(0,0)\in\c^2$. 
In particular, $\pi\colon \c^2_* =\c^2\setminus\{0\} \to \Agot_*$ is a doubly sheeted 
holomorphic covering projection. 

Set $\mu(\zeta,\xi) = \sigma\bigl(\zeta,r(\zeta)\, \xi\bigr)$ and extend it by zero to points $\zeta\in\t\setminus I$.
The conditions on $\sigma$ imply that the partial derivative 
\[
	\sigma_2(\zeta,\xi):= \frac{\di \sigma}{\di \xi}(\zeta,\xi) \in\Agot_*, \quad (\zeta,\xi)\in I\times \cd
\]
has values in  $\Agot_*$.  Note that
$ \frac{\di}{\di \xi} \mu(\zeta,\xi)=r(\zeta) \, \sigma_2(\zeta,r(\zeta)\,\xi)$.
Since $I\times \cd$ is simply connected, there is a 
lifting $\varsigma\colon  I\times \cd \to \c^2_*$ such that $\pi\circ\varsigma=\sigma_2$. Set
\[
	\eta(\zeta,\xi) = \sqrt{r(\zeta)}\, \varsigma \bigl(\zeta,r(\zeta)\, \xi\bigr),  
	\quad (\zeta,\xi)\in \t\times\cd.
\]
Then $\eta(\zeta,\xi)$ is holomorphic in $\xi\in\d$ for every fixed $\zeta\in\t$, and we have that
\[
	\pi\bigl(\eta(\zeta,\xi)\bigr) = r(\zeta) \, \sigma_2(\zeta,r(\zeta)\,\xi) 
	% = \frac{\di}{\di \xi} \sigma(\zeta,r(\zeta)\, \xi)
	= \frac{\di}{\di \xi} \mu(\zeta,\xi). 
\]
We can approximate $\eta$ 
as closely as desired in the sup norm on $\t\times\cd$ by a rational map 
\begin{equation}\label{eq:tilde-eta}
	\wt \eta(\zeta,\xi) = \sum_{j=0}^l B_j(\zeta)\, \xi^j
\end{equation}
where every $B_j(\zeta)$ is a $\c^2$-valued Laurent polynomial with the only pole at $\zeta=0$. Set
\begin{equation}\label{eq:tilde-mu}
	\wt \mu(\zeta,z) = \int_0^z \pi\bigl( \wt\eta(\zeta,\xi)\bigr) \,d\xi = \sum_{k=1}^m A_k(\zeta) \, z^k
\end{equation}
where $\pi$ is the projection (\ref{eq:pi}), $m=2l+1$, 
and $A_k(\zeta)$ are $\c^3$-valued Laurent polynomials with the only pole at $\zeta=0$.
Then $\wt\mu$ is uniformly close to $\mu$ on $\t\times\cd$, and it suffices to prove 
the lemma with $\mu$ replaced by $\wt \mu$. 

To simplify the notation we now drop the tildes and assume that the functions $\eta$ and  $\mu$ 
are given by (\ref{eq:tilde-eta}) and  (\ref{eq:tilde-mu}), respectively. In particular, we have
\[
	\mu_2(\zeta,\xi):= \frac{\di \mu}{\di \xi}(\zeta,\xi) =\pi(\eta(\zeta,\xi)).
\]

%
%  Lemma - Estimate
%
\begin{lemma}\label{lem:estimate}
Let $\mu(\zeta,\xi)= \sum_{k=1}^m A_k(\zeta) \, \xi^k$
where every $A_k(\zeta)$ is a Laurent polynomial with the only pole at $\zeta=0$. 
Write $\mu_2(\zeta,\xi)=\frac{\di \mu}{\di \xi}(\zeta,\xi)$. Then 
\begin{equation}\label{eq:estimate}
	\lim_{N\to\infty} \sup_{|z|\le 1,\, |c|=1} \left | 
	\int_0^z c N\zeta^{N-1} \mu_2(\zeta,c\, \zeta^N)\, d\zeta - \mu(z,c z^N) \right | =0.
\end{equation}
\end{lemma}

\begin{proof}
We have 
$
	\mu_2(\zeta,\xi) = \frac{\di \mu}{\di \xi}(\zeta,\xi) = \sum_{k=1}^m A_k(\zeta) \, k \xi ^{k-1}
$
and hence
\[
  	c N\zeta^{N-1} \mu_2(\zeta,c\, \zeta^N) = \sum_{k=1}^m c^k  A_k(\zeta)  kN\zeta^{kN-1}.
\]
If $N\in\n$ is chosen big enough then $A_k(\zeta)\zeta^N$ vanishes at $\zeta=0$
for every $k=1,\ldots,m$. For such $N$ integration by parts gives 
\begin{equation}\label{eq:IntN}
	\int_0^z A_k(\zeta)  kN\zeta^{kN-1} \,d\zeta = \int_0^z A_k(\zeta)\,d\zeta^{kN} 
	= A_k(z) z^{kN} - \int_0^z A'_k(\zeta) \zeta^{kN}  \,d\zeta.
\end{equation}
Since $A'_k(\zeta)=\sum_{|j| \le m_k} A'_{k,j}\zeta^j$ for some integer $m_k\in\n$, we have
\[
	\int_0^z A'_k(\zeta) \zeta^{kN} d\zeta 
	= \sum_{|j| \le m_k}  \int_0^z  A'_{k,j}\zeta^{j+kN}\,d\zeta
	= \sum_{|j| \le m_k} \frac{A'_{k,j} z^{j+kN+1}}{j+kN+1}.
\]
The right hand side converges to zero uniformly on $\cd= \{|z|\le 1\}$ when
$N\to\infty$. Multiplying the equation (\ref{eq:IntN}) by  $c^k\in \t$,
summing over  $k=1,\ldots,m$ and observing that
\[
	\sum_{k=1}^m c^k A_k(z) z^{kN} = \sum_{k=1}^m A_k(z) (cz^N)^k = \mu(z,c z^N)
\]
we get (\ref{eq:estimate}). This proves Lemma \ref{lem:estimate}.
\end{proof}

For every point $c=\E^{\imath\phi}\in\t$ with $\phi\in [0,2\pi)$ we set 
$\sqrt c = \E^{\imath \phi/2}$. Consider the sequence of maps 
$g_N\colon\c\times\t \to \c^2$ $(N\in\n)$ given by 
\begin{equation}\label{eq:gN}
	g_N(\zeta,c) = \sqrt{c} \, \sqrt{2N+1}\, \zeta^N \eta(\zeta,c\, \zeta^{2N+1}).
\end{equation}
Note that $g_N$ is a holomorphic polynomial in $\zeta\in \c$ for every sufficiently big $N$, say
$N\ge N_0$. Since the projection $\pi$ (\ref{eq:pi}) is homogeneous quadratic, we have that
\[
	\pi(g_N(\zeta,c)) = c (2N+1)\zeta^{2N} \pi\bigl( \eta(\zeta,c\, \zeta^{2N+1})\bigr) 
	= c (2N+1)\zeta^{2N} \mu_2(\zeta,c\, \zeta^{2N+1})
\] 
and hence
\[
	\int_0^z \pi(g_N(\zeta,c)) \,d\zeta = \int_0^z c (2N+1)\zeta^{2N} \mu_2(\zeta,c\, \zeta^{2N+1})\, d\zeta.
\]
By Lemma \ref{lem:estimate} we have
\begin{equation}\label{eq:estimate2}
	\lim_{N\to\infty} \sup_{|z|\le 1,\, c\in\t} \left | 
	\int_0^z \pi\bigl(g_N(\zeta,c)\bigr) d\zeta - \mu(z,c z^{2N+1}) \right | =0.
\end{equation}

The derivative $F' \colon \overline{\mathbb D}\to \Agot^2_*$ of the given null disc $F$ 
lifts to a continuous map $h=(u,v)\colon \cd\to\c^2_*$ that is holomorphic on $\d$. 
With $g_N$ as in (\ref{eq:gN}) we consider the sequence of maps $h_N\colon \cd\times \t \to\c^2$ 
given for $N\ge N_0$ by
\begin{equation}\label{eq:hN}
	h_N(\zeta,c) = h(\zeta) + g_N(\zeta,c),\quad \zeta \in \cd,\ c\in\t.
\end{equation}
A general position argument shows that for a generic choice of $F$ 
we have $h_N(\cd\times \t)\subset \c^2_*$ for all sufficiently big $N\in \n$ (see the proof of Lemma 3.1 in
\cite{AF2} for the details). Assume that this is the case. Consider the holomorphic null discs
\[
	F_N(z,c) = F(0) + \int_0^z \pi(h_N(\zeta,c))\,d\zeta,\quad z\in \cd,\ c\in \t.
\]
Since $\pi$ is a homogeneous quadratic map, we have 
\begin{equation}\label{eq:pi-hN}
	\pi(h_N(\zeta,c))  =  \pi(h(\zeta)) + \pi(g_N(\zeta,c)) + R_N(\zeta,c)
\end{equation}
where each component of  the remainder term $R_N(\zeta,c)$ is a linear combination with constant 
coefficients of terms $g_{N,j}(\zeta,c) u(\zeta)$ and  $g_{N,j}(\zeta,c) v(\zeta)$ for $j=1,2$. 
(Here we write $g_N=(g_{N,1},g_{N,2})$.) We claim that 
\begin{equation}\label{eq:estimate3}
	\lim_{N\to\infty} \sup_{|z|\le 1,\, c\in\t} \left |\int_0^z R_N(\zeta,c) \, d\zeta \right | =0.
\end{equation}
To see this, set
\[
	C_1= \sup_{ |\zeta| \le1, |z|\le 1} |\zeta^{N_0}\eta(\zeta,z)|,\quad 
	C_2=\max\bigl\{ \sup_{ |\zeta| \le1} |u(\zeta)|,\,  \sup_{ |\zeta| \le1}|v(\zeta)| \bigr\}.
\]
Then  $\sup_{ |\zeta| \le1} |\zeta^{N_0} \eta(\zeta,c \zeta^{2N+1})| \le C_1$ for $N\in\n$. 
Given $z\in\cd$, $c\in\t$, $j\in\{1,2\}$ and $N\ge N_0$ we then have 
\begin{eqnarray}\label{eq:est-gN}
	\left| \int_0^z  g_{N,j}(\zeta,c) u(\zeta) d\zeta \right| &\le & 
	 \int_0^{|z|}  \sqrt{2N+1}\,  |\zeta|^{N-N_0}   
	 |\zeta^{N_0}\eta(\zeta,c\, \zeta^{2N+1})| \cdotp |u(\zeta)| \, d|\zeta|  \cr
	 &\le &   C_1C_2 \int_0^{|z|}  \sqrt{2N+1}\,  |\zeta|^{N-N_0} \, d|\zeta|  \cr
	&\le &    C_1C_2 \frac{\sqrt{2N+1}}{N-N_0+1}. %  |z|^{N-N_0+1}.
\end{eqnarray}
Clearly the right hand side converges to zero as $N\to+\infty$. The same estimate holds 
with $u(\zeta)$ replaced by $v(\zeta)$. Since $R_N(\zeta,c)$
is a linear combination of finitely many such terms whose number is independent of $N$,
(\ref{eq:estimate3}) follows.

Since $\pi(h(\zeta))= F'(\zeta)$, we get by integrating the equation (\ref{eq:pi-hN}) and 
using the estimates (\ref{eq:estimate2}), (\ref{eq:estimate3}) that
\begin{equation}\label{eq:FN}
	F_N(z,c) = F(z) + \mu(z,c z^{2N+1}) + E_N(z,c) = \varkappa(z,c z^{2N+1}) + E_N(z,c)  
\end{equation}
where 
\begin{equation}\label{eq:est-EN}
	\lim_{N\to\infty} \sup_{|z|\le 1,\, c\in\t} |E_N(z,c)|=0. 
\end{equation}

It is easily seen  that for every $c\in\t$ and for all sufficiently big $N\in \n$ the null disc 
$G=F_N(\cdotp,c)$ satisfies  conditions {\em i)} -- {\em iii)}  in Lemma \ref{lem:RH3}; a suitable choice of the constant 
$c\in\t$ ensures that it also satisfies condition (\ref{eq:small-increase}). 
(See the proof of \cite[Lemma 3.1]{AF2} and of \cite[Lemma 3.1]{DF2012} for the details.) 
\end{proof}

%
%
%  The case n>3
%
%

We now proceed to the case $n>3$. This requires some additional preparations. 

Let $\bu, \bv, \bw\in \Agot_*$ be linearly independent null vectors such that
\begin{equation}\label{eq:nondeg}
	c:=\Theta(\bu,\bv)\ne 0,\quad b:=\Theta(\bu,\bw)\ne 0,  \quad a:=\Theta(\bv,\bw)\ne 0,
\end{equation}
where $\Theta$ is the complex bilinear form on $\c^n$ given in \eqref{eq:bilinear}.
Denote by $\Agot_{(\bu,\bv,\bw)}$ the intersection of $\Agot$ with the complex 
$3$-dimensional subspace $\Lscr(\bu,\bv,\bw)$ of $\c^n$ spanned by the vectors $\bu,\bv,\bw$.
Condition (\ref{eq:nondeg}) ensures that $\Agot_{(\bu,\bv,\bw)}$
is biholomorphic (in fact, linearly equivalent) to the  $2$-dimensional null quadric $\Agot^2\subset\c^3$. 
Indeed, a calculation shows that $\alpha \bu+\beta \bv+ \gamma \bw\in \Agot$ 
for some $(\alpha,\beta,\gamma)\in \c^3$ if and only if
\[
	\alpha\beta\, \Theta(\bu,\bv)+ \alpha\gamma\, \Theta(\bu,\bw)+\beta\gamma\, \Theta(\bv,\bw)=0.
\] 
Using the notation (\ref{eq:nondeg}), the above equation is equivalent to 
\[
	 \left(\frac \alpha a-\imath \frac \beta b\right)^2+
         \left(\frac \beta b-\imath \frac\gamma c\right)^2
         + \left(\frac\gamma c-\imath \frac \alpha a\right)^2=0.
\]
This is the equation of the null quadric $\Agot^2\subset\c^3$ (\ref{eq:Agot}) in the coordinates
\[
	z_1=\frac \alpha a-\imath \frac \beta b, 
	\quad z_2=\frac \beta b-\imath \frac\gamma c,
	\quad  z_3=\frac\gamma c-\imath \frac \alpha a.
\] 
Note that 
\[
	\bz=(z_1,z_2,z_3)=  (\alpha,\beta,\gamma)\cdotp  A(a,b,c)
\]
where $A(a,b,c)$ is the following nonsingular $3\times 3$ matrix with holomorphic coefficients: 
\begin{equation}\label{eq:A}
	A(a,b,c) =\left(\begin{matrix} 1/a  & 0  & -\imath/a \cr
					-\imath/b & 1/b &  0  \cr
					0  & -\imath/c & 1/c  \cr
				\end{matrix}
	\right).
\end{equation}
We are using row vectors and matrix product on the right for the convenience of notation.

Let $\pi\colon\c^2\to\c^3$ be the homogeneous quadratic map 
given by (\ref{eq:pi}). Note that 
\[
	\pi(1,0)= (1,0,-\imath),\quad \pi(0,1)=(-1,0,-\imath).
\]
Recall that the restriction $\pi\colon \c^2_*=\c^2\setminus\{0\} \to \Agot^2_*$ 
is a doubly sheeted holomorphic covering projection. We have
\begin{eqnarray*}
	\pi\left(\frac 1 {\sqrt a},0\right) &=& \left(\frac 1 a,0,-\frac\imath a\right)
	= (1,0,0) \cdotp A(a,b,c), \cr    
	\pi\left(\imath\sqrt{\frac\imath{2b}},-\sqrt{\frac\imath{2b}}\right) &=& 
	 \left(-\frac\imath b,\frac 1 b,0\right) =  (0,1,0) \cdotp A(a,b,c).
\end{eqnarray*}
The choice of $\sqrt{ }$ is fine on any simply connected subset in the domain space.
Pick a holomorphically varying family of linear automorphisms $\phi_{(a,b)}$ of $\c^2$ such that 
\[
	\phi_{(a,b)}(0,0)=(0,0), \quad \phi_{(a,b)}(1,0)=\left(\frac 1 {\sqrt a},0\right), \quad 
	\phi_{(a,b)}(0,1)=\left(\imath \sqrt{\frac\imath{2b}},-\sqrt{\frac\imath{2b}}\right).
\]
This is achieved by taking $\phi_{(a,b)}(s,t)= (s,t) \cdotp B(a,b)$ where $B$ is the $2\times 2$ matrix
\begin{equation}\label{eq:B}
	B(a,b) =\left(\begin{matrix} \frac{1}{\sqrt a}  & 0  \cr
					\imath \sqrt{\frac\imath{2b}}  & -\sqrt{\frac\imath{2b}} \cr
			\end{matrix}	\right).
\end{equation}
The map 
\[
	\c^2\ni  (s,t) \mapsto  
	\bigl(\alpha(s,t), \beta(s,t),\gamma(s,t)\bigr)=
	\pi\bigl((s,t)\cdotp B(a,b)\bigr)  \cdotp A(a,b,c)^{-1} \in \c^3
\]
is homogeneous quadratic in $(s,t)$ and depends holomorphically on $(a,b,c)$, 
and hence on the triple $(\bu,\bv,\bw)$ of null vectors satisfying (\ref{eq:nondeg}).  
By the construction the associated map
\begin{equation}\label{eq:psi}
	\c^2\ni  (s,t) \mapsto  
	\psi_{(\bu,\bv,\bw)}(s,t) =\alpha(s,t) \bu+\beta(s,t) \bv+ \gamma (s,t)\bw 
\end{equation}
is a holomorphically varying parametrization of the quadric $\Agot_{(\bu,\bv,\bw)}$ satisfying 
\begin{equation}\label{eq:normalization}
 	\psi_{(\bu,\bv,\bw)}(\be_1) =\bu, \quad  \psi_{(\bu,\bv,\bw)}(\be_2) =\bv
\end{equation}
where $\be_1=(1,0)$ and $\be_2=(0,1)$. 
Note that $\psi_{(\bu,\bv,\bw)}$ is well defined on the set of triples $(\bu,\bv,\bw)\in (\Agot^{n-1})^3$ 
satisfying condition (\ref{eq:nondeg}), except for the indeterminacies caused by the square roots
in the entries of the matrix $B$ (\ref{eq:B}). 
These reflect the fact that $\pi$ is a doubly sheeted quadratic map, so
we have  four different choices providing the normalization (\ref{eq:normalization}).

In the sequel we shall hold fixed a pair of null vectors $\bu,\bv\in \Agot_*$, subject to the condition 
$c=\Theta(\bu,\bv)\ne 0$, and will assume that $\bw=f(\zeta)$ where 
$f\colon \overline{\mathbb D}\to \Agot_*$ is a holomorphic map such that the triple 
of null vectors $(\bu,\bv,f(\zeta))$ satisfies (\ref{eq:nondeg}) 
for every $\zeta\in \cd$. 

The following lemma provides an approximate solution of the Riemann-Hilbert problem
for null holomorphic discs in $\c^n$ for any $n\ge 3$ under the condition that the null discs
attached at boundary points of the (arbitrary) center null disc $F$ have a constant
direction vector $\bu\in \Agot_*$. (For $n=3$ this result is subsumed by Lemma \ref{lem:RH3}.)
At this time we are unable to prove the exact analogue of Lemma \ref{lem:RH3}
in dimensions $n>3$, but the present version is entirely sufficient for the 
applications in this paper.

%
%
%  
%
%  RH PROBLEM FOR NULL DISCS IN C^n
%
%
%
%
\begin{lemma}  \label{lem:RH}
Fix an integer $n\ge 3$ and let $\Agot=\Agot^{n-1}$ be the null quadric (\ref{eq:Agot}).
Assume that $\bu,\bv\in\Agot_*=\Agot\setminus\{0\}$ are null vectors such that 
$\Theta(\bu,\bv)\ne0$. Let $F\colon\overline{\d}\to\c^n$ be a null holomorphic 
immersion of class $\Ascr^1(\d)$ whose derivative $f=F'\colon \overline{\mathbb D}\to \Agot_*$ 
satisfies the following nondegeneracy condition: 
\begin{equation}\label{eq:nondeg2}
	\Theta(\bu,f(\zeta))\ne 0\ \ \text{and} \ \  \Theta(\bv,f(\zeta))\ne 0
	\quad \forall \zeta\in\cd.
\end{equation}
Let $r\colon \t=b\d \to \r_+ := [0,+\infty)$ be a continuous function (the {\em size function}),  
and let $\sigma \colon \t \times\overline{\d}\to\c$ be a function of class $\Cscr^1$ such that
for every $\zeta\in \t$ the function $\cd\ni \xi \mapsto \sigma(\zeta,\xi)$ is holomorphic on $\d$,
$\sigma(\zeta,0)=0$, the partial derivative $\frac{\di \sigma}{\di \xi}$ is nonvanishing on $\t\times \cd$, 
and the winding number of the function $\t \ni \zeta \mapsto \frac{\di \sigma}{\di \xi}(\zeta,0) \in \c_*$ 
equals zero.  Set $\mu(\zeta,\xi) = r(\zeta) \sigma(\zeta,\xi)$ and let 
$\varkappa\colon \t \times\overline{\d}\to\c^n$ be given by
\begin{equation}\label{eq:varkappa2}
	\varkappa(\zeta,\xi)=F(\zeta) + \mu(\zeta,\xi)\, \bu = F(\zeta)+ r(\zeta)\, \sigma(\zeta,\xi)\, \bu.
\end{equation}
Given numbers $\epsilon>0$ and $0<\rho_0<1$, there exist a number $\rho'\in [\rho_0,1)$ 
and a null holomorphic disc $G\colon\overline{\d}\to\c^n$ %of class $\Ascr^1(\d)$ 
such that $G(0)=F(0)$ and the following conditions hold:
\begin{enumerate}[\it i)]
\item $\dist(G(\zeta),\varkappa(\zeta,\t))<\epsilon$ for all $\zeta\in \t$, 
\item $\dist(G(\rho\zeta),\varkappa(\zeta,\overline{\d}))<\epsilon$ for all $\zeta\in \t$ and all 
$\rho\in [\rho',1)$, and 
\item $G$ is $\epsilon$-close to $F$ in the $\Cscr^1$ topology on $\{\zeta\in\c\colon |\zeta|\leq \rho'\}$.
\end{enumerate}
Furthermore, if $I$ is a compact arc in $\t$ such that the function $r$ vanishes on $\t\setminus  I$
and $U$ is an open neighborhood of $I$ in $\overline{\d}$, then in addition to the above
\begin{enumerate}[\it i)]
\item[\it iv)] one can choose $G$ to be $\epsilon$-close to $F$ in the $\Cscr^1$ topology 
on $\overline{\d}\setminus  U$.
\end{enumerate}
Moreover, given an upper semicontinuous function $\phi \colon\c^n\to \r\cup\{-\infty\}$ 
and a closed arc $I\subset \t$, we may achieve in addition  that 
\begin{equation}
\label{small-increase2}
	\int_{I} \phi \bigl( G(\E^{\imath t})\bigr) \, \frac{dt}{2\pi} \le 
 	\int^{2\pi}_0 \!\! \int_{I} \phi \bigl(\varkappa(\E^{\imath t},\E^{\imath s})\bigr) 
       \frac{dt}{2\pi} \frac{ds}{2\pi} + \epsilon
\end{equation}
\end{lemma}

\begin{remark}
For every $\zeta\in\t$ the map $\cd\ni \xi \mapsto \sigma(\zeta,\xi)\bu$ in the lemma is an 
immersed holomorphic disc directed by the null vector $\bu\in\Agot_*$; the nonnegative function
$r(\zeta)\ge 0$ is used to rescale it. If the support of $r$ is contained in a proper subarc $I$ of the circle 
$\t$ then it suffices to assume that $\sigma(\zeta,\xi)$ is defined for $\zeta\in I$, and in this 
case the winding number condition on the function $ \frac{\di \sigma}{\di \xi}(\zeta,\xi)\ne 0 $ is irrelevant.
Conditions (\ref{eq:small-increase}) and (\ref{small-increase2}) are not used in this paper, but 
they will be used in the envisioned applications to minimal hulls. 
\end{remark}

\begin{proof}
Write $\Agot=\Agot^{n-1}$ and fix a pair of null vectors
$\bu,\bv\in\Agot_*$ as in the lemma. Given a vector $\bw\in \Agot_*$ such that the triple 
$(\bu,\bv,\bw)$ satisfies condition (\ref{eq:nondeg}), we denote by 
$\psi_\bw \colon \c^2\to \Agot$ the map $\psi_{(\bu,\bv,\bw)}$ (\ref{eq:psi});
hence by (\ref{eq:normalization}) we have that 
\[
	\psi_\bw (\be_1) = \bu,\quad \psi_\bw (\be_2) =\bv.
\]
Recall that $f=F'\colon \cd \to \Agot_*$ is a map of class $\Ascr(\d)$, and
condition (\ref{eq:nondeg2}) implies that the triple of null vectors $(\bu,\bv,f(\zeta))$
satisfies condition (\ref{eq:nondeg}) for every $\zeta\in \cd$.
Due to simple connectivity of the disc the coefficients of the matrix function $B$
(\ref{eq:B}) are well defined functions of class $\Ascr(\d)$, and there is a holomorphic map 
$h=(u,v)\colon \cd\to \c^2_*$ satisfying
\begin{equation}\label{eq:h}
	\psi_{f(\zeta)}(h(\zeta))= f(\zeta),\quad \zeta\in\cd.
\end{equation}
The conditions on the functions $r\ge 0$ and $\sigma$ ensure that the partial derivative
\[
	\mu_2(\zeta,\xi):= 
	\frac{\di \mu}{\di \xi}(\zeta,\xi) = r(\zeta)\, \frac{\di \sigma}{\di \xi}(\zeta,\xi)
\]
admits a continuous square root 
\[
	\eta(\zeta,\xi) = \sqrt{\mu_2(\zeta,\xi)},  \qquad (\zeta,\xi)\in \t\times\cd
\]
that is holomorphic in $\xi\in\d$ for every fixed $\zeta\in\t$. We can approximate $\eta$ 
as closely as desired in the sup norm on $\t\times\cd$ by a rational function of the form
\begin{equation}\label{eq:tilde-eta2}
	\tilde\eta(\zeta,\xi) = \sum_{j=0}^l B_j(\zeta)\, \xi^j
\end{equation}
where every $B_j(\zeta)$ is a Laurent polynomial with the only pole at $\zeta=0$. Set
\begin{equation}\label{eq:tilde-mu2}
	\tilde\mu(\zeta,z) = \int_0^z \tilde\eta(\zeta,\xi)^2 \,d\xi 
	= \sum_{k=1}^m A_k(\zeta) \, z^k
\end{equation}
where $m=2l+1$ and $A_k(\zeta)$ are Laurent polynomials with the only pole at $\zeta=0$.
Then $\tilde\mu$ is uniformly close to $\mu$ on $\t\times\cd$, and it suffices to prove 
the lemma with $\mu$ replaced by $\tilde \mu$. 
We now drop the tildes and assume that the functions $\eta$ and  $\mu$ 
are given by (\ref{eq:tilde-eta2}) and  (\ref{eq:tilde-mu2}), respectively.
For every point $c=\E^{\imath\phi}\in\t$ with $\phi\in [0,2\pi)$ we set 
$\sqrt c = \E^{\imath \phi/2}$ and consider the sequence of functions 
\[
	g_N(\zeta,c) = \sqrt{c} \, \sqrt{2N+1}\, \zeta^N \eta(\zeta,c\, \zeta^{2N+1}).
\]
(Formally this coincides with (\ref{eq:gN}), except that $g_N$ is now scalar-valued.)
Note that $g_N$ is a holomorphic polynomial in $\zeta\in \c$ for every sufficiently big $N$ and 
\[
	\int_0^z g_N(\zeta,c)^2 \,d\zeta = \int_0^z c (2N+1)\zeta^{2N} \mu_2(\zeta,c\, \zeta^{2N+1})\, d\zeta
\]
where $\mu_2(\zeta,\xi)=\frac{\di \mu}{\di \xi}(\zeta,\xi)$.  
By Lemma \ref{lem:estimate} we have
\begin{equation}\label{eq:estimate-gN}
	\lim_{N\to\infty} \sup_{|z|\le 1,\, c\in\t} \left | 
	\int_0^z g_N(\zeta,c)^2 \,d\zeta - \mu(z,c z^{2N+1}) \right | =0.
\end{equation}
For every sufficiently big $N \in \n$ we define the map 
$h_N=(u_N,v_N)\colon \cd\times\t\to\c^2$ by 
\[
	h_N(\zeta,c)= h(\zeta) +  g_N(\zeta,c) \,\be_1 = \bigl( u(\zeta) + g_N(\zeta,c), v(\zeta) \bigr).
\] 
By general position, moving $f$ and hence $h$ slightly, we can ensure that
$h_N(\zeta,c)\ne (0,0)$ for all $(\zeta,c)\in \cd\times\t$ and all sufficiently big $N\in \n$. 
Indeed, we claim that this holds as long as $v$ has no zeros on $\t\times \t$. Under this
assumption we have that $v(\zeta)\ne 0$ in an annulus $\rho_1 \le |\zeta|\le 1$ 
for some $\rho_1 <1$, whence $h_N(\zeta,c) \ne (0,0)$ for such $\zeta$ and for any $N\in \n$
and $c\in\t$.  Since $h_N \to h$ uniformly on 
$\{|\zeta|\le \rho_1\}\times\t$ as $N\to +\infty$
and $h$ does not assume the value $(0,0)$, the same is true for $h_N$ 
for all sufficiently big $N\in \n$. 

By the definition of $\psi_{f(\zeta)}$  the map
\[
 	f_N(\zeta,c) := \psi_{f(\zeta)}\bigl(h_N(\zeta,c)\bigr) \in \c^n, 
 	\quad (\zeta,c)\in \cd\times\t
\]
has range in the punctured null quadric $\Agot_*$, and hence the map
\[
	\cd\ni z\mapsto F_N(z,c) = F(0)+ \int_0^z f_N(\zeta,c)\, d\zeta \in \c^n
\]
is an immersed holomorphic null disc for every $c\in\t$ and for all sufficiently big $N\in \n$.

We claim that if $N\in\n$ is chosen big enough then the null disc 
$G=F_N(\cdot,c)$ satisfies Lemma \ref{lem:RH} 
for a suitable choice of the constant $c=c_N \in\t$.  
Indeed, since all maps in the definition of $\psi_{f(\zeta)}$ are
either linear (given by a product with the matrices $A^{-1}$ and $B$)
or homogeneous quadratic (the projection $\pi$ given by (\ref{eq:pi})), we infer that
\begin{equation}\label{eq:psi-f}
	\psi_{f(\zeta)} \left( g_N(\zeta,c) \be_1\right) =
	g_N(\zeta,c)^2 \, \bu,  \quad (\zeta,c)\in \cd\times\t.
\end{equation}
As we also have $\psi_{f(\zeta)}(h(\zeta))=f(\zeta)$ by (\ref{eq:h}), we get for
$(\zeta,c)\in \cd\times\t$ that
\begin{equation}\label{eq:fN}
	f_N(\zeta,c) = \psi_{f(\zeta)}(h_N(\zeta,c)) = f(\zeta) + g_N(\zeta,c)^2 \, \bu + R_N(\zeta,c).
\end{equation}
In order to estimate the remainder $R_N(\zeta,c)$
we observe that the terms in $f_N(\zeta,c)$ are of three different kinds as follows:
\begin{itemize} 
\item[\rm (a)] Terms which contain $u^2$, $v^2$ or $uv$ (where $h=(u,v)$);
the sum of all such terms equals  $f(\zeta)$  in view of (\ref{eq:h}).
\item[\rm (b)] Terms  which do not contain any  component $u,v$ of $h$; 
the sum of all such terms equals $g_N(\zeta,c)^2 \, \bu$ in view of (\ref{eq:psi-f}).
\item[\rm (c)] Terms which contain exactly one component $u,v$ of $h$
and exactly one copy of the function $g_N(\zeta,c)$. 
All such terms are placed in the remainder $R_N$.
\end{itemize} 
The terms described above are multiplied by various elements of the matrices $A^{-1}$ (\ref{eq:A}) and 
$B$ (\ref{eq:B}); those are  functions in $\Ascr(\d)$ depending on $f$ and $h$
but not depending on $N$. By integrating the equation (\ref{eq:fN})
we see that the map $F_N(z,c)$ is of the form (\ref{eq:FN}) and the remainder 
$E_N(z,c)$ satisfies (\ref{eq:est-EN}) as is seen from the estimates (\ref{eq:estimate-gN}) and (\ref{eq:est-gN}).  
\end{proof}

By using Lemmas \ref{lem:RH3} and \ref{lem:RH} we can prove the following result which gives an approximate 
solution to the Riemann-Hilbert problem for bordered Riemann surfaces as null curves
in $\c^n$. The case $n=3$ corresponds to \cite[Theorem 3.4]{AF2}. 
The proof given there extends directly to the present situation, replacing 
\cite[Lemma 3.1]{AF2} with Lemma \ref{lem:RH} above.

%
%
%  RH PROBLEM FOR NULL CURVES IN C^n
%
%
\begin{theorem}[{\bf Riemann-Hilbert problem for null curves in $\c^n$}]\label{th:RH}
Fix an integer $n\ge 3$ and let $\Agot=\Agot^{n-1}\subset\c^n$ denote the null quadric (\ref{eq:Agot}).
Let $M$ be a compact bordered Riemann surface with boundary $bM\ne\emptyset$, and let $I_1,\ldots,I_k$ be a finite collection of  pairwise disjoint compact subarcs of $bM$ which are not connected components of $bM$. Choose a thin annular neighborhood  
$A\subset M$ of $bM$ and a smooth retraction $\rho\colon A\to bM$. Assume that
\begin{itemize}
\item $F\colon M \to\c^n$ is a null holomorphic immersion of class $\Ascr^1(M)$,
\item $\bu_1,\ldots,\bu_k \in \Agot_*=\Agot\setminus\{0\}$ are null vectors (the {\em direction vectors}), 
\item  $r \colon bM \to \r_+$ is a continuous  nonnegative function supported on  $I:=\bigcup_{i=1}^k I_i$, 
\item  $\sigma \colon I \times\overline{\d}\to\c$ is a function of class $\Cscr^1$ such that
for every $\zeta\in I$ the function $\cd\ni \xi \mapsto \sigma(\zeta,\xi)$ is holomorphic on $\d$,
$\sigma(\zeta,0)=0$, and the partial derivative $\frac{\di \sigma}{\di \xi}$ is nowhere vanishing on $I \times \cd$.
\end{itemize}
Consider the continuous map  $\varkappa\colon bM \times\overline{\d}\to\c^n$ given by
\[
	\varkappa(\zeta ,\xi)=\left\{\begin{array}{ll}
	F(\zeta );     & \zeta \in bM\setminus I \\    
	F(\zeta ) + r(\zeta) \sigma(\zeta,\xi) \bu_i; & \zeta \in I_i,\; i\in\{1,\ldots,k\}.
	\end{array}\right.
\]
Given a number $\epsilon>0$ there exists an arbitrarily small open neighborhood $\Omega\subset M$ of   
$I=\bigcup_{i=1}^k I_i$ and a null holomorphic immersion $G\colon M\to\c^n$ of class $\Ascr^1(M)$ 
satisfying the following properties:
\begin{enumerate}[\it i)]
\item $\dist(G(\zeta ),\varkappa(\zeta ,\t))<\epsilon$ for all $\zeta \in bM$.
\item $\dist(G(\zeta ),\varkappa(\rho(\zeta ),\overline{\d}))<\epsilon$ for all $\zeta \in \Omega$.
\item $G$ is $\epsilon$-close to $F$ in the $\Cscr^1$ topology on $M\setminus \Omega$.
\end{enumerate}
\end{theorem}

\begin{proof} 
Recall that $\Theta$ is the bilinear form (\ref{eq:bilinear}) on $\c^n$.
For each index $i\in \{1,\ldots,k\}$ we choose a null vector $\bv_i\in \Agot_*$
such that $\Theta(\bu_i,\bv_i)\ne 0$. Pick a holomorphic $1$-form $\theta$ without zeros on $M$. 
Then $dF=f\theta$ where $f\colon M\to \Agot_*$ is a map of class $\Ascr(M)$.
Deforming $F$ slightly if needed we may assume that $f$ is nondegenerate 
(see Sec.\ \ref{sec:prelim}). By a small deformation of the pairs of null vectors $(\bu_i,\bv_i)$  
we may also assume that the following conditions hold on each of the arcs $I_i$ (cf.\ (\ref{eq:nondeg2})):
\[
	\Theta(\bu_i,f(\zeta ))\ne 0\quad \text{and}\quad \Theta(\bv_i,f(\zeta ))\ne 0 \quad \text{for all } \zeta\ \in I_i.
\]
By continuity there is a neighborhood  $U_i\subset A\subset M$ of the arc $I_i$ such that 
the same conditions hold for all $\zeta \in U_i$.

By \cite[Lemma 5.1]{AF1} there exists a {\em spray of maps} $f_w\colon M\to \Agot_*$ of class $\Ascr(M)$, 
depending holomorphically on a parameter $w$ in a ball $B\subset \c^N$ 
for some big integer $N$ and satisfying the following properties:
\begin{itemize}
\item[\rm (a)] the spray is {\em dominating}, i.e., the partial differential 
$ %\[
	\frac{\di}{\di w}\big|_{w=0} f_w(\zeta ) \colon \c^N\to \c^n
$ %\]
is surjective for every $\zeta \in M$. 
\item[\rm (b)] The spray is {\em period dominating} in the following sense.
Let $C_1,\ldots, C_l$ be smooth closed curves in $\mathring M$ which form the basis
of the homology group $H_1(M;\z)$. Then the period map
$\Pcal=(\Pcal_1,\ldots,\Pcal_l)\colon B \to (\c^n)^l$ with the components
\[
	\Pcal_j(w) = \int_{C_j} f_w\theta \in \c^n, \quad w\in B,\ j=1,\ldots,l
\] 
has the property that $\frac{\di}{\di w}\big|_{w=0} \Pcal(w): \c^N\to (\c^n)^l$ is surjective.
\end{itemize}

For each $i=1,\ldots, l$ we pick a compact, smoothly bounded, simply connected domain
$D_i$ in $M$ (a disc) such that $D_i\subset U_i$, $D_i$ contains a neighborhood of the arc $I_i$
in $M$, and $D_i\cap D_j=\emptyset$ for $1\le i\ne j\le l$.
Since the curves $C_j$ lie in the interior of $M$, we can choose the $D_i$'s small enough such that
$\bigcup_{j=1}^l C_j \cap \bigcup_{i=1}^k D_i=\emptyset$. For every $i=1,\ldots,l$ 
the function $\sigma(\zeta,\xi)$ can be extended to $(\zeta,\xi) \in bD_i\times \cd$ such that conditions of 
Lemma \ref{lem:RH} are fulfilled on the disc $D_i$, 
and the function $r$ extends to $bD_i$ such that it vanishes on $bD_i\setminus I_i$.

Under these conditions we can apply Lemma \ref{lem:RH} on each disc $D_i$ 
to approximate the restricted spray $f_w|_{D_i}$ as closely as desired,
uniformly on $D_i\setminus V_i$ for a small neighborhood $V_i\subset D_i$ of the arc $I_i$, 
by a holomorphic spray $g_{i,w}\colon D_i\to\Agot_*$ of class $\Ascr(D_i)$ such that the integrals
$G_{i,w}(\zeta ) = \int^\zeta  g_{i,w}\theta$ $(\zeta \in D_i,\ w\in B)$ 
with suitably chosen initial values at some point $p_i\in D_i$ satisfy the conclusion of Lemma \ref{lem:RH}.
(To be precise, we need a parametric version of Lemma \ref{lem:RH} with a holomorphic dependence
on the parameter $w$ of the spray. It is clear that the proof of Lemma \ref{lem:RH} 
gives this without any changes.)

Assuming that these approximations are close enough, the domination property (a) of the 
spray $f_{w}$ allow us to glue the sprays $f_w$ and $g_{i,w}$ for $i=1,\ldots, l$
into a new holomorphic spray $\tilde g_w\colon M\to \Agot_*$  which approximates 
$f_w$ very closely on $M\setminus \bigcup_{i=1}^l V_i$. (The parameter ball $B\subset\c^N$ 
shrinks a little. For the details and references regarding this gluing see  \cite[Theorem 3.4]{AF2}.) 
The period domination property of $f_w$ (condition (b) above) implies
that there exists  $w_0\in B$ close to $0$ such that the map
$g=\tilde g_{w_0}\colon M\to \Agot_*$ has vanishing periods over the curves $C_j$.
The holomorphic map $G\colon M\to \c^n$, defined by 
$G(\zeta ) = F(p)+\int_p^\zeta  g\theta$ $(\zeta \in M$)  for a fixed initial point $p\in M$, 
is then a null holomorphic immersion satisfying Theorem \ref{th:RH} provided that all 
approximations were sufficiently close.
\end{proof}

By adapting Theorem \ref{th:RH} to  conformal minimal immersions we obtain the following. 

%
%
%  RH PROBLEM FOR CONFORMAL MININAL IN C^n
%
%
\begin{theorem}[\bf Riemann-Hilbert problem for conformal minimal immersions in $\r^n$] 
\label{th:RHCMI}
Assume that $n\ge 3$ and the data $M$,  $I_1,\ldots,I_k\subset bM$, $I=\bigcup_{i=1}^k I_i$, $r \colon bM \to \r_+$,
$\sigma\colon I \times \cd\to\c$, $A\subset M$ and $\rho\colon A\to bM$ are as in Theorem \ref{th:RH}.
Let $F\in\CMI^1(M,\r^n)$. For each $i=1,\ldots, k$ let $\bu_i,\bv_i\in \r^n$ be a pair of orthogonal vectors 
satisfying $\|\bu_i\|=\|\bv_i\|>0$. Consider the continuous map  
$\varkappa\colon bM \times\overline{\d}\to\r^n$ given by 
\[
	\varkappa(\zeta ,\xi)=\left\{\begin{array}{ll}
	F(\zeta ),  & \zeta \in bM\setminus I; \\ 
	F(\zeta ) + r(\zeta )\bigl( \Re \sigma(\zeta,\xi) \bu_i+ \Im \sigma(\zeta,\xi)  \bv_i \bigr), & 
	\zeta \in I_i,\; i\in\{1,\ldots,k\}.
	\end{array}\right.
\]
Given a number $\epsilon>0$ there exist an arbitrarily small open neighborhood $\Omega\subset M$ of   
$I=\bigcup_{i=1}^k I_i$ and a  conformal minimal immersion $G\in\CMI^1(M,\r^n)$
satisfying  the following properties:
\begin{enumerate}[\it i)]
\item $\dist(G(\zeta ),\varkappa(\zeta ,\t))<\epsilon$ for all $\zeta \in bM$.
\item $\dist(G(\zeta ),\varkappa(\rho(\zeta ),\overline{\d}))<\epsilon$ for all $\zeta \in \Omega$.
\item $G$ is $\epsilon$-close to $F$ in the $\Cscr^1$ norm on $M\setminus \Omega$.
\item $\Flux (G)=\Flux (F)$.
\end{enumerate}
\end{theorem}

\begin{proof}
The conditions on $\bu_i,\bv_i\in\r^n$ imply that 
$\wt \bu_i=\bu_i-\imath \bv_i\in\Agot_*$ is a null vector for every $i=1,\ldots,k$. 
Pick a holomorphic $1$-form $\theta$ without zeros on $M$. Then
$\di F=f\theta$ where $f\colon M\to \Agot_*$ is a holomorphic map of class $\Ascr(M)$. 
We apply the proof of Theorem \ref{th:RH} to the map $f$
and the null vectors $\wt \bu_i$, but with the following difference.
At the very last step of the proof we can argue that we obtain a holomorphic 
map $g=\tilde g_{w_0}\colon M\to \Agot_*$ in the new spray 
(for some value of the parameter $w_0\in \c^N$ close to $0$) such that
\[
	\int_{C_j} g\theta = \int_{C_j} f\theta = \int_{C_j}\di F;  \quad j=1,\ldots,l.
\] 
Since $\int_{C_j}2\Re (\di F) = \int_{C_j} dF=0$, the real periods of $g$ vanish 
and the imaginary periods equal those of $f$. Hence the map  $G\colon M\to\r^n$,
given by $G(\zeta )=F(p_0)+\int_{p_0}^\zeta  \Re(g\theta)$ 
for some fixed $p_0\in M$, is a conformal minimal immersion with $\Flux(G)=\Flux(F)$
(property {\it iv)}). Properties  {\it i)--iii)} of $G$ follow  from 
the corresponding properties of the map $\int^\zeta  g\theta$ 
on each disc $D_i\subset M$ constructed in the proof of Theorem \ref{th:RH}.
(In fact, with a correct choice of initial values we have $G(\zeta ) =\int^\zeta  \Re(g\theta)$ 
for $\zeta \in D_i$, $i=1,\ldots,k$.)
\end{proof}

%%%%%%%%%%
%%%%%%%%%%
%%%%%%%%%% Jordan curves
%%%%%%%%%%
%%%%%%%%%%
%%%%%%%%%%

\section{Complete minimal surfaces bounded by Jordan curves}\label{sec:Jordan}

In this section we prove Theorem \ref{th:Jordan}. The key in the proof  is the following lemma
which asserts that every conformal minimal immersion $M\to\r^n$ of a compact bordered Riemann surface
can be approximated as close as desired in the $\Cscr^0(M)$ topology by conformal minimal immersions 
with arbitrarily large intrinsic diameter.

\begin{lemma}\label{lem:Jordan2}
Let $M$ be a compact bordered Riemann surface, let $n\geq 3$ be a natural number, and let $G\in\CMI_*^1(M,\r^n)$. 
Given a point $p_0\in \mathring M$ and a number $\lambda>0$, we can approximate $G$ arbitrarily closely 
in the $\Cscr^0(M)$ topology by a conformal minimal immersion $\wh G\in\CMI_*^1(M,\r^n)$ such that 
$\dist_{\wh G}(p_0,bM)>\lambda$ and $\Flux(\wh G)=\Flux(G)$.
\end{lemma}

The notation $\CMI_*^1(M)$ has been introduced in Sec.\ \ref{sec:prelim}. Since conformal minimal immersions are 
harmonic, the approximation in the above lemma takes place in the $\Cscr^r$ topology on compact subsets of $\mathring M$ 
for all $r\in\z_+$. However, if $\lambda>\dist_G(p_0,bM)$ then the approximation in the $\Cscr^1(M)$ topology is clearly impossible. 

%%%%%%%%%%
%%%%%%%%%%
%%%%%%%%%%
%%%%%%%%%%
%%%%%%%%%%
%%%%%%%%%%

Lemma \ref{lem:Jordan2} will follow by a standard recursive application of the following result.

\begin{lemma}\label{lem:Jordan1}
Let $M$ be a compact bordered Riemann surface and let $n\geq 3$. 
Consider $F\in\CMI^1(M,\r^n)$, a smooth map $\Ygot\colon bM\to\r^n$, and a number $\delta>0$ such that
\begin{equation}\label{eq:lemJordan1}
\|F-\Ygot\|_{0,bM}<\delta. 
\end{equation}
Fix a point $p_0\in \mathring M$ and choose a number  $d>0$ such that
\begin{equation}\label{eq:d}
	0<d<\dist_F(p_0,bM).
\end{equation}
Then for each $\eta>0$ the map $F$ can be approximated uniformly on compacts in $\mathring M$ by nondegenerate
conformal minimal immersions $\wh F\in\CMI_*^1(M,\r^n)$ satisfying the following properties:
\begin{enumerate}[\rm (a)]
\item $\|\wh F-\Ygot\|_{0,bM}<\sqrt{\delta^2+\eta^2}$.
\item $\dist_{\wh F}(p_0,bM)> d+\eta$.
\item $\Flux(\wh F)=\Flux(F)$.
\end{enumerate}
\end{lemma}

The key idea in the proof of Lemma \ref{lem:Jordan1} is to push the $F$-image of each point $p\in bM$ 
a distance approximately $\eta$ in a direction approximately orthogonal to $F(p)-\Ygot(p)\in\r^n$. 
Conditions {\rm (a)} and {\rm (b)} will then follow from \eqref{eq:lemJordan1}, \eqref{eq:d}, and Pythagoras' Theorem.

The main improvement of Lemmas \ref{lem:Jordan2} and \ref{lem:Jordan1} with respect to previous related results in the literature 
is that we do not change the conformal structure on the source bordered Riemann surface $M$. This particular point is the key that 
allows us to ensure that the complete minimal surfaces constructed in Theorem \ref{th:Jordan} are bounded by Jordan curves.

\begin{proof}[Proof of Lemma \ref{lem:Jordan1}]
We assume that $M$ is a smoothly bounded compact domain in an open Riemann surface $\wt M$. Furthermore, 
in view of the Mergelyan theorem for conformal minimal immersions into $\r^n$ 
(Theorems 3.1 and 5.3 in \cite{AFL}) we may also assume that $F$ extends to $\wt M$ 
as a conformal minimal immersion in $\CMI_*(\wt M,\r^n)$.

Fix a number $\epsilon>0$ and a compact set $K\Subset M$. 
To prove the lemma, we will construct a nondegenerate conformal minimal immersion 
$\wh F\in\CMI_*^1(M,\r^n)$  which is $\epsilon$-close to $F$ in the $\Cscr^1(K)$ norm and 
satisfies conditions {\rm (a)}, {\rm (b)}, and {\rm (c)}.

Enlarging $K$ if necessary we assume that $K$ is a smoothly bounded compact domain in $\mathring M$ 
which is a strong deformation retract of $M$, that $p_0\in\mathring K$, and (see \eqref{eq:d}) that
\begin{equation}\label{eq:>tau}
	\dist_F(p_0,bK)>d.
\end{equation}
By general position we may also assume  that 
\begin{equation}\label{eq:neq}
	\text{the map $F-\Ygot$ does not vanish anywhere on $bM$.}
\end{equation}

Denote by $\alpha_1,\ldots,\alpha_k$ the connected components of $bM$
and recall that every $\alpha_i$ is a smooth Jordan curve in $\wt M$. 

Fix a number $\epsilon_0>0$ which will be specified later. 

By \eqref{eq:lemJordan1} and the continuity of $F$ and $\Ygot$ there exist a natural number $l\geq 3$ and compact connected 
subarcs $\{\alpha_{i,j}\subset\alpha_i\colon (i,j)\in \ttI:=\{1,\ldots,k\}\times\z_l\}$ (here $\z_l=\z/l\z$) such that
\begin{equation}\label{eq:alpha}
	\bigcup_{j\in\z_l}\alpha_{i,j}=\alpha_i,  \quad i\in\{1,\ldots,k\},
\end{equation} 
and, for every $(i,j)\in \ttI$, 
the arcs $\alpha_{i,j}$ and $\alpha_{i,j+1}$ have a common endpoint $p_{i,j}$ and are otherwise disjoint,
$\alpha_{i,j}\cap\alpha_{i,a}=\emptyset$ for all $a\in\z_l\setminus\{j-1,j,j+1\}$,
\begin{equation}\label{eq:Y-Y}
	\text{$\|\Ygot(p)-\Ygot(q)\|<\epsilon_0$ for all $\{p,q\}\subset \alpha_{i,j}$,}
\end{equation}
and 
\begin{equation}\label{eq:F-Yalpha}
	\text{$\|F(p)-\Ygot(q)\|<\delta$ and $\|F(p)-F(q)\|<\epsilon_0$ for all $\{p,q\}\subset \alpha_{i,j}$}.
\end{equation}

For every $(i,j)\in \ttI$ we denote by 
\begin{equation}\label{eq:piij}
	\pi_{i,j}\colon\r^n\to {\rm span}\bigl\{F(p_{i,j})-\Ygot(p_{i,j})\bigl\}\, \subset\r^n
\end{equation}
the orthogonal projection onto the affine real line ${\rm span}\{F(p_{i,j})-\Ygot(p_{i,j})\}$; cf.\ \eqref{eq:neq}.

The first main step in the proof consists of perturbing $F$ near the points $\{p_{i,j}\colon (i,j)\in \ttI\}$ in order to find a conformal 
minimal immersion in $\CMI_*^1(M,\r^n)$ which is close to $F$ in the $\Cscr^1(K)$ topology and the distance between $p_0$ and 
$\{p_{i,j}\colon (i,j)\in \ttI\}$ in the induced metric is large in a suitable way. This deformation procedure is enclosed in the following result.

\begin{lemma}\label{lem:step1}
Given a number $\epsilon_1>0$, there exists a nondegenerate conformal minimal immersion 
$F_0\in\CMI_*^1(M,\r^n)$ satisfying the following properties:
\begin{enumerate}[\rm ({P}1)]
\item $F_0$ is $\epsilon_1$-close to $F$ in the $\Cscr^1(K)$ topology.
\item $\|F_0(p)-\Ygot(q)\|<\delta$ and $\|F_0(p)-F(q)\|<\epsilon_0$ for all $\{p,q\}\subset\alpha_{i,j}$, 
for all $(i,j)\in \ttI$.
\item $\Flux(F_0)=\Flux(F)$.
\item For every $(i,j)\in I$ there exists a small open neighborhood $U_{i,j}$ of $p_{i,j}$ in $M$, with $\overline U_{i,j}\cap K=\emptyset$, fulfilling the following condition: If $\gamma\subset M$ is an arc with initial point in $K$ and final point in $\overline U_{i,j}$, 
and if $\{J_{a,b}\}_{(a,b)\in \ttI}$ is any partition of $\gamma$ by Borel measurable subsets, then
\[
	\sum_{(a,b)\in \ttI} \length \, \pi_{a,b}(F_0(J_{a,b})) >\eta,
\]
where $\eta>0$ is the real number given in the statement of Lemma \ref{lem:Jordan1} and $\pi_{a,b}$ are the 
projections in \eqref{eq:piij}, $(a,b)\in \ttI$.
\end{enumerate}
\end{lemma}

\begin{proof}  
Choose a family of pairwise disjoint Jordan arcs $\{\gamma_{i,j}\subset \wt M\colon (i,j)\in \ttI\}$ such that
each $\gamma_{i,j}$ contains $p_{i,j}$ as an endpoint, is attached transversely to $M$ at $p_{i,j}$, 
and is otherwise disjoint from $M$. The set 
\begin{equation}\label{eq:S}
	S:=M\cup\bigcup_{(i,j)\in \ttI} \gamma_{i,j} \subset\wt M
\end{equation}
is {\em admissible} in $\wt M$ the sense of \cite[Def.\ 5.1]{AFL}. 
Take a smooth map $u\colon S\to\r^n$ satisfying the following properties:
\begin{enumerate}[\rm (i)]
\item $u=F$ on a neighborhood of $M$.
\item $\|u(x)-\Ygot(q)\|<\delta$  and $\|u(x)-F(q)\|<\epsilon_0$ for all 
$(x,q)\in (\gamma_{i,j-1}\cup \alpha_{i,j}\cup \gamma_{i,j})\times \alpha_{i,j}$, for all $(i,j)\in \ttI$.
\item If $\{J_{a,b}\}_{(a,b)\in \ttI}$ is a partition of $\gamma_{i,j}$ by Borel measurable subsets, then
\[
\sum_{(a,b)\in \ttI} \length\, \pi_{a,b}(u(J_{a,b})) > 2\eta.
\]
\end{enumerate}
Notice that condition \eqref{eq:F-Yalpha} allows one to choose a map $u$ satisfying {\rm (i)} and {\rm (ii)}.  
To ensure also {\rm (iii)}, one can simply choose $u$ over each arc $\gamma_{i,j}$ to be highly oscillating in the direction 
of $F(p_{a,b})-\Ygot(p_{a,b})$ for all $(a,b)\in\ttI$, but with sufficiently small extrinsic diameter so that {\rm (ii)} 
remains to hold. (Recall \eqref{eq:neq} and \eqref{eq:piij} and take into account that $\ttI$ is finite.)

Let $\theta$ be a nowhere vanishing holomorphic $1$-form on $\wt M$ (such exists by the Oka-Grauert principle; 
cf.\ Theorem 5.3.1 in \cite[p.\ 190]{F2011}). It is then easy to find a smooth function $f\colon S\to\Agot_*^{n-1}$, 
holomorphic in a neighborhood of $M$, 
such that the pair $(u,f\theta)$ is a {\em generalized conformal minimal immersion} on the set $S$ (\ref{eq:S})
in the sense of \cite[Def.\ 5.2]{AFL}. 

Fix a number $\epsilon_2>0$ which will be specified later.

Since $M$ is a strong deformation retract of $S$ (\ref{eq:S}) and taking {\rm (i)} into account, 
the Mergelyan theorem for conformal minimal immersions into $\r^n$ \cite[Theorem 5.3]{AFL} furnishes a 
conformal minimal immersion $G\in\CMI_*^1(\wt M,\r^n)$ such that
\begin{enumerate}[\rm ({P}1)]
\item[\rm (iv)] $G$ is $\epsilon_2$-close to $u$ in the $\Cscr^0(S)$ and the $\Cscr^1(M)$ topologies, and
\item[\rm (v)] $\Flux(G)=\Flux(F)$.
\end{enumerate}

Let $V\subset \wt  M$ be a small open neighborhood of $S$. For every $(i,j)\in \ttI$ let $q_{i,j}$ denote 
the endpoint of $\gamma_{i,j}$ different from $p_{i,j}$.  If $\epsilon_2>0$ is small enough, 
then properties {\rm (ii)} and {\rm (iii)} guarantee the existence of small open neighborhoods 
$W'_{i,j}\Subset W_{i,j}$ of $p_{i,j}$ and $V_{i,j}$ of $\gamma_{i,j}$ in $V \setminus K$, $(i,j)\in \ttI$, 
satisfying the following conditions:
\begin{enumerate}[\rm ({P}1)]
\item[\rm (vi)] $V_{i,j}\cap M\Subset W'_{i,j}\Subset W_{i,j}\Subset V\setminus K$.
\item[\rm (vii)] $\|G(x)-\Ygot(q)\|<\delta$ and $\|G(x)-F(q)\|<\epsilon_0$ for all 
$(x,q)\in (W_{i,j-1}\cup V_{i,j-1}\cup \alpha_{i,j}\cup W_{i,j}\cup V_{i,j})\times \alpha_{i,j}$, for all $(i,j)\in \ttI$.
\item[\rm (viii)] If $\gamma'_{i,j}\subset W_{i,j}\cup V_{i,j}$ is an arc with initial point in $W_{i,j}$ and final point $q_{i,j}$, 
and if $\{J_{a,b}\}_{(a,b)\in \ttI}$ is a partition of $\gamma'_{i,j}$ by Borel measurable subsets, then
\[
	\sum_{(a,b)\in  \ttI} \length\, \pi_{a,b}(G(J_{a,b})) >2\eta.
\]
\end{enumerate}
Without loss of generality we assume in addition that the compact sets $\overline W_{i,j}\cup\overline V_{i,j}$, $(i,j)\in \ttI$, 
are pairwise disjoint.    

By \cite[Theorem 2.3]{FW0} (see also \cite[Theorem 8.8.1]{F2011}) there exists a smooth diffeomorphism $\phi\colon M\to\phi (M)\subset V$ satisfying the following properties:
\begin{enumerate}[\rm (i)]
\item[\rm (ix)] $\phi \colon  \mathring M \to  \phi(\mathring M)$ is a biholomorphism.
\item[\rm (x)] $\phi$ is as close as desired to the identity in the $\Cscr^1$ topology on $M \setminus \bigcup_{(i,j)\in \ttI} W'_{i,j}$. 
\item[\rm (xi)] $\phi(p_{i,j}) = q_{i,j}\in b\,\phi(M)$ and $\phi(M \cap W'_{i,j}) \subset W_{i,j}\cup V_{i,j}$ for all $(i,j)\in \ttI$.
\end{enumerate}

Let us check that, if $\epsilon_2>0$ is sufficiently small and if the approximation in {\rm (x)} is close enough, 
then the conformal minimal immersion
\begin{equation}\label{eq:G}
	F_0:=G\circ \phi\in \CMI_*^1(M,\r^n)
\end{equation}
satisfies the conclusion of Lemma \ref{lem:step1}. 

Indeed, property {\rm (vi)} gives that 
\begin{equation}\label{eq:K}
K\Subset M \setminus \bigcup_{(i,j)\in \ttI} W_{i,j}\Subset M \setminus \bigcup_{(i,j)\in \ttI} W'_{i,j}.
\end{equation}
Therefore {\rm (i)}, {\rm (iv)}, and {\rm (x)} ensure that $F_0$ is $\epsilon_1$-close to $F$ in the $\Cscr^1(K)$ topology 
provided that $\epsilon_2$ is small enough, thereby proving {\rm (P1)}. 

In order to check {\rm (P2)} fix $(i,j)\in\ttI$ and take $\{p,q\}\in\alpha_{i,j}$. If $p\in \alpha_{i,j} \setminus (W'_{i,j-1}\cup W'_{i,j})$, 
then $F_0(p)\approx G(p)$ by \eqref{eq:G} and {\rm (x)}, hence {\rm (vii)} gives that $\|F_0(p)-\Ygot(q)\|<\delta$ and 
$\|F_0(p)-F(q)\|<\epsilon_0$. If $p\in W'_{i,j-1}\cup W'_{i,j}$, then {\rm (xi)} gives that 
$\phi(p)\in W_{i,j-1}\cup V_{i,j-1}\cup W_{i,j}\cup V_{i,j}$, and so \eqref{eq:G} and {\rm (vii)} imply that 
$\|F_0(p)-\Ygot(q)\|<\delta$ and $\|F_0(p)-F(q)\|<\epsilon_0$ as well. This proves {\rm (P2)}.

 Property {\rm (P3)} is directly implied by {\rm (v)} and \eqref{eq:G}.

Finally, in order to check {\rm (P4)} fix $(i,j)\in\ttI$, let $\gamma\subset M$ be an arc with the initial point in $K$ 
and the final point $p_{i,j}$, and let $\{J_{a,b}\}_{(a,b)\in\ttI}$ be a partition of $\gamma$ by Borel measurable subsets. 
Properties \eqref{eq:K} and {\rm (x)} give that $\phi(K)\Subset V \setminus \bigcup_{(i,j)\in \ttI} W_{i,j}$. 
Properties  {\rm (vi)}, {\rm (x)}, and {\rm (xi)} guarantee that $\phi(\gamma)$ has a connected subarc contained in 
$W_{i,j}\cup V_{i,j}$ with the initial point in $W_{i,j}$ and the final point $q_{i,j}$. Therefore, \eqref{eq:G} and 
condition {\rm (viii)} trivially implies the existence of neighborhoods $U_{i,j}$ of $p_{i,j}$ satisfying  {\rm (P4)}.
This proves Lemma \ref{lem:step1}.
\end{proof}

We continue with the proof of Lemma \ref{lem:Jordan1}.

Let $F_0\in\CMI_*^1(M,\r^n)$ and $\{U_{i,j}\colon (i,j)\in \ttI\}$  be furnished by Lemma \ref{lem:step1} for a given 
number $\epsilon_1>0$ that will be specified later. Up to a shrinking, we may assume that the sets $\overline U_{i,j}$, 
$(i,j)\in \ttI$, are simply connected, smoothly bounded, and pairwise disjoint.
Roughly speaking, $F_0$ meets conditions {\rm (a)} and {\rm (c)} in Lemma \ref{lem:Jordan1} (cf.\ properties {\rm (P2)} 
and {\rm (P3)}), but it satisfies condition {\rm (b)} only on the sets $bM\cap \overline U_{i,j}$, $(i,j)\in \ttI$ (cf.\ {\rm (P4)}). 
To conclude the proof we now perturb $F_0$ near the points of $bM$ where it does not meet  {\rm (b)} 
(i.e.,\ outside $\bigcup_{(i,j)\in \ttI} U_{i,j}$), preserving what has already been achieved so far. 
It is at this stage where the Riemann-Hilbert problem for minimal surfaces in $\r^n$ (see Theorem \ref{th:RHCMI}) 
will be exploited.

Let $\epsilon_3>0$ be a positive number that will be specified later.

Take an annular neighborhood $A\subset M\setminus K$ of $bM$ and a smooth retraction $\rho\colon A\to bM$.
In view of condition {\rm (P2)} we may choose a family of pairwise disjoint, smoothly bounded closed disc 
$\overline D_{i,j}$ in $M\setminus K$, $(i,j)\in\ttI$, satisfying
\begin{equation}\label{eq:F0D}
	\text{$\|F_0(p)-\Ygot(q)\|<\delta$\ for all $(p,q)\in \overline D_{i,j}\times\alpha_{i,j}$}
\end{equation}
and  the following properties:
\begin{enumerate}[\it i)]
\item $\bigcup_{(i,j)\in \ttI} \overline D_{i,j}\subset A$.
\item $\overline D_{i,j}\cap bM$ is a compact connected Jordan arc in $\alpha_{i,j}\setminus\{p_{i,j-1},p_{i,j}\}$ with an endpoint 
in $U_{i,j-1}$ and the other endpoint in $U_{i,j}$.
\item $\rho(\overline D_{i,j})\subset \alpha_{i,j}\setminus\{p_{i,j-1},p_{i,j}\}$ and $\|F_0(\rho(x))-F_0(x)\|<\epsilon_3$ for all 
$x\in\overline D_{i,j}$ and all $(i,j)\in \ttI$.
\end{enumerate}

For each $(i,j)\in \ttI$ we also choose a pair of compact connected Jordan arcs 
$\beta_{i,j}\Subset I_{i,j}\Subset \overline D_{i,j}\cap\alpha_{i,j}$ with an endpoint in $U_{i,j-1}$ and the other endpoint in 
$U_{i,j}$, and a pair of vectors $\bu_{i,j}$, $\bv_{i,j}\in \r^n$, such that 
\begin{equation}\label{eq:uijvij}
	\|\bu_{i,j}\|=1 = \|\bv_{i,j}\|;  \quad \text{$\bu_{i,j}$, $\bv_{i,j}$, and $F(p_{i,j})-\Ygot(p_{i,j})$ are pairwise orthogonal}.
\end{equation}
Let $\mu\colon bM\to\r_+$ be a continuous function such that
\begin{equation}\label{eq:mu}
	0\leq\mu\leq\eta,\quad \text{$\mu=\eta$\ \ on  $\bigcup_{(i,j)\in I} \beta_{i,j}$,\quad \text{$\mu=0$\; 
	on $bM\setminus 		\bigcup_{(i,j)\in \ttI} I_{i,j}$}}.
\end{equation}
Consider the continuous map  $\varkappa\colon bM \times\overline{\d}\to\r^n$ given by 
\begin{equation}\label{eq:vark}
	\varkappa(x,\xi)=\left\{\begin{array}{ll}
	F_0(x), & x\in bM\setminus \bigcup_{(i,j)\in \ttI} I_{i,j}; \\
	F_0(x) + \mu(x)\,(\Re \xi \bu_{i,j} + \Im \xi  \bv_{i,j}),  & x\in I_{i,j},\; (i,j)\in\ttI.
	\end{array}\right.
\end{equation}

In this setting, Theorem \ref{th:RHCMI} provides  for every $(i,j)\in \ttI$ an arbitrarily small 
open neighborhood $\Omega_{i,j}\subset \overline D_{i,j}$ of   
$I_{i,j}$ in $M$ and a nondegenerate conformal minimal immersion $\wh F\in \CMI_*^1(M,\r^n)$ 
satisfying the following properties:
\begin{enumerate}[\rm ({P}1)]
\item[\rm (P5)] $\dist(\wh F(x),\varkappa(x,\t))<\epsilon_3$ for all $x\in bM$.
\item[\rm (P6)] $\dist(\wh F(x),\varkappa(\rho(x),\overline{\d}))<\epsilon_3$ for all 
$x\in \Omega:=\bigcup_{(i,j)\in\ttI}\Omega_{i,j}$.
\item[\rm (P7)] $\wh F$ is $\epsilon_3$-close to $F_0$ in the $\Cscr^1$ topology on $M\setminus \Omega$.
\item[\rm (P8)] $\Flux (\wh F)=\Flux (F_0)$.
\end{enumerate}
Recall that $\pi_{i,j}$ is the projection \eqref{eq:piij}. 
Note  that {\rm (P6)}, \eqref{eq:vark}, and property {\it iii)} ensure that
\begin{enumerate}[\rm ({P}1)]
\item[\rm (P9)] $\pi_{i,j}\circ\wh F$ is $2\epsilon_3$-close to $\pi_{i,j}\circ F_0$ in the $\Cscr^0(\Omega_{i,j})$ topology
for all $(i,j)\in\ttI$. 
\end{enumerate}

Let us check that $\wh F$ satisfies the conclusion of Lemma \ref{lem:Jordan1} 
provided that the positive numbers $\epsilon_0$, $\epsilon_1$, and $\epsilon_3$ are chosen sufficiently small. 

Notice that properties  {\rm (P1)} and {\rm (P7)} imply that 
\begin{equation}\label{eq:close}
\text{$\wh F$ is $(\epsilon_1+\epsilon_3)$-close to $F$ in the $\Cscr^1(K)$ topology},
\end{equation}
and hence $\wh F$ and $F$ are $\epsilon$-close in $\Cscr^1(K)$ provided that $\epsilon_1+\epsilon_3<\epsilon$; 
take into account that $\Omega\subset\bigcup_{(i,j)\in I} \overline D_{i,j}\subset A\subset M\setminus K$.

Let us now check property {\rm (a)} in  Lemma \ref{lem:Jordan1}. 
Fix a point $p\in bM$. If  $p\in bM\setminus \Omega$ then 
by {\rm (P7)} we have  $\|\wh F(p)- F_0(p)\|<\epsilon_3$, 
and hence {\rm (P2)} ensures that $\|\wh F(p)-\Ygot(p)\|<\sqrt{\delta^2+\eta^2}$ 
provided that $\epsilon_3>0$ is small enough.

Assume now that $p\in bM\cap\Omega$; then $p\in bM\cap \Omega_{i,j}$ for some $(i,j)\in \ttI$. 
In view of {\rm (P5)}, \eqref{eq:mu}, and \eqref{eq:vark}, we have that
\begin{equation}\label{eq:whFp}
\big\|\wh F(p)- \big(F_0(p) + \mu(p)\,(\Re \xi \bu_{i,j} + \Im \xi  \bv_{i,j})\big)\big\|<\epsilon_3\quad \text{for some $\xi\in\t$.}
\end{equation}
On the other hand, taking into account \eqref{eq:uijvij}, we obtain
\begin{multline*}
\big\|\big(F_0(p) + \mu(p)\,(\Re \xi \bu_{i,j} + \Im \xi  \bv_{i,j})\big) -\Ygot(p)\big\|  \leq
\\ 
\|F_0(p)-F(p_{i,j})\| + \sqrt{\|F(p_{i,j})-\Ygot(p_{i,j})\|^2 + \mu(p)^2}
 +  \|\Ygot(p_{i,j})-\Ygot(p)\| \stackrel{\text{{\rm (P2)}},\eqref{eq:Y-Y}}{<}
\\
 \sqrt{\|F(p_{i,j})-\Ygot(p_{i,j})\|^2 + \mu(p)^2}+2\epsilon_0.
\end{multline*}
Together with \eqref{eq:whFp}, \eqref{eq:mu}, and \eqref{eq:lemJordan1} we get
\[
\|\wh F(p)-\Ygot(p)\|<\sqrt{\|F(p_{i,j})-\Ygot(p_{i,j})\|^2 + \mu(p)^2}+2\epsilon_0+\epsilon_3<\sqrt{\delta^2+\eta^2},
\]
where the latter inequality holds provided that $\epsilon_0$ and $\epsilon_3$ are chosen small enough from the beginning. 
This proves property {\rm (a)} in Lemma \ref{lem:Jordan1}.

Let us now verify property  {\rm (b)}. Recall that $p_0\in \mathring K$.  If $\epsilon_1$ and $\epsilon_3$ are small enough, 
then \eqref{eq:>tau} and \eqref{eq:close} ensure that
\begin{equation}\label{eq:distwhF}
	\dist_{\wh F}(p_0,bK)>d.
\end{equation}
We now estimate $\dist_{\wh F}(bK,bM)$. Properties {\rm (P4)}, {\rm (P7)}, and {\rm (P9)} guarantee the following:

\begin{claim}\label{cla:ep3}
If $\epsilon_3>0$ is chosen small enough, then for every arc $\gamma\subset M\setminus \mathring K$ with the 
initial point in $bK$ and the final point in $\bigcup_{(i,j)\in\ttI}\overline U_{i,j}$, 
and for any partition $\{J_{a,b}\}_{(a,b)\in\ttI}$ of $\gamma$ by Borel measurable subsets satisfying
$\gamma\cap\overline \Omega_{a,b}\subset J_{a,b}$ for all $(a,b)\in\ttI$,
we have 
\[
	\length\, \wh F(\gamma) \geq \sum_{(a,b)\in\ttI}\length \, \pi_{a,b}(\wh F(J_{a,b})) >\eta.
\]
\end{claim}

Consider now an arc $\gamma\subset M\setminus \bigcup_{(i,j)\in\ttI} \overline U_{i,j}$ with the initial point in $bK$, 
the final point in $bM$, and otherwise disjoint from $K$. 
Then there exist $(i,j)\in\ttI$ and a subarc $\hat\gamma$ of $\gamma$ with the endpoints
$q\in M\setminus\Omega$ and $p\in \beta_{i,j}$ 
satisfying $\hat\gamma\subset \overline\Omega_{i,j}\setminus (\overline U_{i,j-1}\cup \overline U_{i,j})$.
In view of \eqref{eq:uijvij}, \eqref{eq:mu} and \eqref{eq:whFp} there exists $\xi\in\t$ such that
\begin{equation}\label{eq:whF-F0}
\|\wh F(p)- F_0(p)\|> \mu(p)\|\Re \xi \bu_{i,j} + \Im \xi  \bv_{i,j}\|-\epsilon_3=\eta-\epsilon_3.
\end{equation}
On the other hand, we have
\begin{eqnarray*}
	\length\, \wh F(\gamma) & \geq & \length\, \wh F(\hat \gamma) \ge  \|\wh F(q)-\wh F(p)\| 
\\
 & \geq & \|\wh F(p)-F_0(p)\| 
\\
 & & - \|\wh F(q)-F_0(q)\| - \|F_0(q)-F_0(\rho(q))\| - \|F_0(\rho(q))-F_0(p)\|
\\
 & > & \|\wh F(p)-F_0(p)\| -\epsilon_0-2\epsilon_3,
\end{eqnarray*}
where in the last inequality we used {\rm (P2)}, {\it iii)}, and {\rm (P7)}. Combining this inequality and \eqref{eq:whF-F0} 
we get that  $\length\, \wh F(\gamma)>\eta-\epsilon_0-3\epsilon_3$. Together with Claim \ref{cla:ep3} we obtain that 
$\dist_{\wh F}(bK,bM)>\eta-\epsilon_0-3\epsilon_3$ and, taking into account  \eqref{eq:distwhF}, $\dist_{\wh F}(p_0,bM)>d+\eta$ 
provided that $\epsilon_0$ and $\epsilon_3$ are small enough. This shows property {\rm (b)}.

Finally, condition {\rm (c)} is trivially implied by  {\rm (P3)} (cf.\ Lemma \ref{lem:step1})
and {\rm (P8)}, thereby concluding the proof of Lemma \ref{lem:Jordan1}.
\end{proof}

%%%%%%%%%%
%%%%%%%%%%
%%%%%%%%%%
%%%%%%%%%%
%%%%%%%%%%
%%%%%%%%%%

\begin{proof}[Proof of Lemma \ref{lem:Jordan2}]
Let $\epsilon>0$. We shall find $\wh G\in\CMI_*^1(M,\r^n)$ which is $\epsilon$-close to $G$ in the $\Cscr^0(M)$ topology and 
satisfies $\dist_{\wh G}(p_0,bM)>\lambda$ and $\Flux(\wh G)=\Flux(G)$.

Choose numbers $d_0$ and $\delta_0$ such that $0<d_0<\dist_G(p_0,bM)$
and $0<\delta_0<\epsilon$. Set 
\[
	c:=\frac{\sqrt{6(\epsilon^2-\delta_0^2)}}{\pi}>0.
\]
Consider the following sequences  defined recursively:
\[
d_j:=d_{j-1}+\frac{c}{j}>0,\qquad \delta_j:=\sqrt{\delta_{j-1}^2+\frac{c^2}{j^2}}>0,\quad j\in\n.
\]
Observe that
\begin{equation}\label{eq:limits}
\{d_j\}_{j\in\z_+}\nearrow +\infty,\qquad \{\delta_j\}_{j\in\z_+}\nearrow \epsilon.
\end{equation}
We claim that there exists a sequence $G_j\in\CMI_*^1(M,\r^n)$
$(j\in\z_+)$ of conformal minimal immersions  enjoying the following properties:
\begin{enumerate}[\rm (a$_{j}$)]
\item $\|G_j-G\|_{0,bM}<\delta_j$.
\item $\dist_{G_j}(p_0,bM)>d_j$.
\item $\Flux(G_j)=\Flux(G)$.
\end{enumerate}
We proceed by induction, beginning with the immersion $G_0:=G$. For the inductive step we assume the existence of 
$G_j\in\CMI_*^1(M,\r^n)$ satisfying {\rm (a$_j$)}, {\rm (b$_j$)}, and {\rm (c$_j$)} for some $j\in\z_+$. 
Applying Lemma \ref{lem:Jordan1} to the data
\[
F=G_j,\quad \Ygot=G|_{bM},\quad \delta=\delta_j,\quad p_0,\quad \eta=\frac{c}{j+1},\quad d=d_j,
\]
we obtain a conformal minimal immersion $G_{j+1}\in\CMI_*^1(M,\r^n)$ satisfying 
{\rm (a$_{j+1}$)}, {\rm (b$_{j+1}$)}, and {\rm (c$_{j+1}$)}, hence closing the induction step.

By {\rm (a$_j$)}, the Maximum Principle, and the latter assertion in \eqref{eq:limits}, $G_j$ is $\epsilon$-close to $G$ in the 
$\Cscr^0(M)$ topology for all $j\in \z_+$. On the other hand, {\rm (b$_j$)} and the former assertion in \eqref{eq:limits} ensure that 
$\dist_{G_j}(p_0,bM)>d_j>\lambda$ for any large enough $j\in\z_+$. In view of {\rm (c$_j$)}, to conclude the proof it suffices 
to choose $\wh G:=G_j$ for a sufficiently large $j\in\n$.
\end{proof}

%%%%%%%%%%
%%%%%%%%%%
%%%%%%%%%%
%%%%%%%%%%
%%%%%%%%%%
%%%%%%%%%%

Another important point in the proof of Theorem \ref{th:Jordan} is that the general position of 
conformal minimal immersions $M\to\r^n$ is embedded if $n\geq 5$ (cf.\ \cite[Theorem 1.1]{AFL}). 
Moreover, it is easy to derive from the proof in \cite{AFL} that the general position of the boundary curves 
of conformal minimal immersions $M\to\r^n$ is also embedded for any $n\geq 3$. 
The following is the precise result that will be used in the proof of Theorem \ref{th:Jordan}.

\begin{theorem}\label{th:gp}
Let $M$ be a compact bordered Riemann surface and let $n\geq 3$ and $r\ge 1$ be natural numbers. 
\begin{enumerate}[\rm (a)]
\item Every conformal minimal immersion $F\in \CMI^r(M,\r^n)$ can be approximated in the $\Cscr^r(M)$ topology 
by nondegenerate conformal minimal immersions $\wt F\in \CMI_*^r(M,\r^n)$ such that 
$\wt F|_{bM}\colon bM\to\r^n$ is an embedding and $\Flux(\wt F)=\Flux(F)$.
\item If $n\geq 5$ then every nondegenerate conformal minimal immersion $F\in \CMI_*^r(M,\r^n)$ can be approximated in the 
$\Cscr^r(M)$ topology by nondegenerate conformal minimal embeddings $\wt F\in \CMI_*^r(M,\r^n)$ satisfying 
$\Flux(\wt F)=\Flux(F)$.
\end{enumerate}
\end{theorem}

As already said, assertion {\rm (b)} in this theorem is proved in \cite[Theorem 4.1]{AFL}.

\begin{proof}[Proof of {\rm (a)}]
By \cite[Theorem 3.1 (a)]{AFL} we may assume that  $F\in \CMI_*^r(M,\r^n)$ is nondegenerate.
We consider the {\em difference map} $\delta F\colon M\times M\to \r^n$, defined by
\[
	\delta F(x,y)=F(y)-F(x), \qquad x,y\in M.
\]
Clearly $F$ is injective if and only if $(\delta F)^{-1}(0)= D_M:=\{(x,x): x\in M\}$,
the diagonal of $M\times M$. Since $F$ is an immersion, it is locally injective, and hence there is an open 
neighborhood  $U\subset M\times M$ of $D_M$ 
such that $\delta F$ does not vanish anywhere on $\overline U\setminus D_M$. 

In this setting, the construction in \cite[Sec.\ 4]{AFL} furnishes a neighborhood $\Omega \subset \r^N$  
of the origin in a Euclidean space and a real analytic map $H\colon \Omega \times M \to \r^n$ satisfying 
the following conditions:
\begin{enumerate}[\rm (i)]
\item $H(0,\cdotp)=F$.
\item $H(\xi,\cdotp)\in \CMI_*^r(M)$ and $\Flux(H(\xi,\cdotp))=\Flux(F)$ for every $\xi \in \Omega$.
\item The difference map $\delta H\colon \Omega \times M\times M \to \r^n$, defined by 
\[
	\delta H(\xi,x,y) = H(\xi,y)-H(\xi,x), \qquad \xi\in \Omega, \ \ x,y\in M, 
\]
is a submersive family  on $(M\times M) \setminus U$, in the sense that the partial differential 
\[
	\di_\xi|_{\xi=0} \, \delta H(\xi,x,y) \colon \r^N \to \r^n
\]
is surjective for every $(x,y)\in (M\times M) \setminus U$. 
\end{enumerate}
Set $\psi_H:=H|_{\Omega\times bM}$ and $\delta \psi_H:=(\delta H)|_{\Omega\times bM\times bM}$. 
From {\rm (iii)} and the compactness of $(bM\times bM)\setminus U$ we obtain that the partial differential 
$\di_\xi (\delta\psi_H(\xi,x,y))\colon\r^N\to\r^n$ is surjective for all $\xi$ in a neighborhood 
$\Omega'\subset\Omega$ of $0\in\r^N$, for every $(x,y)\in (bM\times bM)\setminus U$. 
This implies that the map $\delta \psi_H\colon \Omega'\times (bM\times bM)\setminus U\to\r^n$ 
is transverse to any submanifold of $\r^n$, in particular, to the origin $0\in\r^n$. 
The standard transversality argument due to Abraham \cite{Abraham}
(see also \cite[Sec.\ 7.8]{F2011}) ensures that for a generic choice of $\xi\in\Omega'$  the difference map 
$\delta \psi_H(\xi,\cdotp,\cdotp)$ is  transverse to $0\in\r^n$ on $(bM\times bM)\setminus U$.
Since $n\geq 3$ and $(bM\times bM)\setminus U$ is of real dimension 2, it follows that
\begin{equation}\label{eq:psiH}
\text{$\delta \psi_H(\xi,\cdotp,\cdotp)$ does not vanish anywhere on $(bM\times bM)\setminus U$.}
\end{equation}
If we choose $\xi={\xi_0}$ close enough to $0\in\r^N$ and such that \eqref{eq:psiH} holds,
then the conformal minimal immersion $\wt F=H({\xi_0},\cdotp)\in\CMI_*^r(M,\r^n)$ satisfies the conclusion of 
Theorem \ref{th:gp} and is arbitrarily close to $F$ in the $\Cscr^r(M)$ topology.
(See the proof of \cite[Theorem 4.1]{AFL} for further details.)
\end{proof}

%%%%%%%%%%
%%%%%%%%%%
%%%%%%%%%%
%%%%%%%%%%
%%%%%%%%%%
%%%%%%%%%%

% We are now ready to prove the main result in this section.

\begin{proof}[Proof of Theorem \ref{th:Jordan}]
We may assume that $M$ is a smoothly bounded compact domain in an open Riemann surface 
$\wt M$.  By \cite[Theorem 3.1 (a)]{AFL} we can assume that $F$ is nondegenerate, $F\in\CMI_*^1(M,\r^n)$.
By Theorem \ref{th:gp} above we may also assume that $F|_{bM}$ is an embedding and, 
if $n\geq 5$, that $F\colon M\to\r^n$ is an embedding.

Choose a compact domain $M_0\subset\mathring M$, a point $p_0\in \mathring M_0$, and set $F_0:=F$. Let $\theta$ be a 
holomorphic $1$-form in $\wt M$ vanishing nowhere on $M$ and denote by ${\rm d}\colon M\times M\to \r$ the distance function 
on the Riemannian surface $(M,|\theta|^2)$.  

Pick a number $\epsilon_0>0$. 

By applying Lemma \ref{lem:Jordan2} and Theorem \ref{th:gp} we shall inductively construct a sequence 
$\{\Xi_j=(M_j,\epsilon_j,F_j)\}_{j\in\n}$, where $M_j$ is a compact domain in $\mathring M$,
$\epsilon_j>0$, and $F_j\in \CMI_*^1(M,\r^n)$, satisfying
the following properties for all $j\in\n$:
\begin{enumerate}[\rm (1$_{j}$)]
\item $M_{j-1} \Subset \mathring M_{j}$ and $\max\{{\rm d}(p,bM)\colon p\in bM_j\}<1/j$.
\item $\max\big\{\|F_j-F_{j-1}\|_{0,M}, \big\|(\partial F_j-\partial F_{j-1})/\theta\big\|_{0,M_{j-1}}\big\}<\epsilon_j$.
\item $\dist_{F_j}(p_0,bM_k)>k$ for all $k\in\{0,\ldots,j\}$.
\item $\Flux(F_j)=\Flux(F)$.
\item $F_j|_{bM}$ is an embedding and, if $n\geq 5$, $F_j$ is an embedding.
\item $\epsilon_j<\min\left\{ \epsilon_{j-1}/2 , \tau_j,\varsigma_j\right\},$
where the numbers $\tau_j$ and $\varsigma_j$ are defined as follows:
\end{enumerate}
\begin{equation}\label{eq:tauj}
\tau_j=\frac1{2^j}\min_{k\in\{0,\ldots,j-1\}} \min_{p\in M} \left\|\frac{\partial F_k}{\theta}(p) \right\|, 
\end{equation}
\begin{equation}\label{eq:varsigmaj}
\varsigma_j=\left\{
\begin{array}{ll}
\displaystyle \frac1{2j^2} \inf\left\{ \|F_{j-1}(p)-F_{j-1}(q)\|\colon p,q\in bM,\, {\rm d}(p,q)>\frac1{j}\right\} & \text{if $n\in\{3,4\}$}
\\
 & 
\\
\displaystyle \frac1{2j^2} \inf\left\{ \|F_{j-1}(p)-F_{j-1}(q)\|\colon p,q\in M,\, {\rm d}(p,q)>\frac1{j}\right\} & \text{if $n\geq 5$.}
\end{array}
\right.
\end{equation}

Notice that $\Xi_0=(M_0,\epsilon_0,F_0)$ meets conditions {\rm (3$_0$)}, {\rm (4$_0$)}, and {\rm (5$_0$)}, whereas {\rm (1$_0$)}, 
{\rm (2$_0$)}, and {\rm (6$_0$)} are void. Let $j\in\n$ and assume inductively the existence of triples $\Xi_0,\ldots, \Xi_{j-1}$ 
satisfying the above conditions. Since $F_0,\ldots,F_{j-1}$ are immersions, the number $\tau_j$ \eqref{eq:tauj} is positive. 
Moreover, {\rm (5$_{j-1}$)} ensures that the number  $\varsigma_j$ \eqref{eq:varsigmaj} is positive as well. 
Therefore there exists $\epsilon_j>0$ satisfying {\rm (6$_j$)}.  Lemma \ref{lem:Jordan2} ensures that $F_{j-1}$ can be approximated 
in the $\Cscr^0(M)$ topology (and hence also in the $\Cscr^1(M_{j-1})$ topology) by a conformal minimal immersion 
$F_j\in\CMI_*^1(M,\r^n)$ satisfying $\dist_{F_j}(p_0,bM)>j$ and {\rm (4$_j$)}.
Taking into account {\rm (3$_{j-1}$)}, we may choose an $F_j$ with these properties and a compact region 
$M_j\subset \mathring M$ satisfying also  {\rm (1$_j$)}, {\rm (2$_j$)}, and {\rm (3$_j$)}. 
Furthermore, in view of Theorem \ref{th:gp}, we may also assume that $F_j$ meets condition {\rm (5$_j$)}. 
This concludes the inductive step and hence the construction of the sequence $\{\Xi_j\}_{j\in\n}$. 

By properties {\rm (1$_j$)} and {\rm (6$_j$)}, which hold for all $j\in\n$, we have $\bigcup_{j=1}^\infty M_j=M$
and the sequence $\{F_j\}_{j\in\n}$ converges uniformly on $M$ to a continuous map   
\[
	\wt F:=\lim_{j\to\infty} F_j\colon M\to\r^n
\]
which is $\epsilon_0$-close to $F$ in the $\Cscr^0(M)$ topology and 
whose restriction to $\mathring M$ is conformal and harmonic. 
To finish the proof, it remains to show that 
\begin{enumerate}[\rm (a)]
\item $\wt F|_{\mathring M}\colon \mathring M\to\r^n$ is a complete immersion,
\item $\Flux(\wt F)=\Flux(F)$, 
\item $\wt F|_{bM}\colon bM\to\r^n$ is injective, and
\item if $n\geq 5$ then $\wt F\colon M\to\r^n$ is injective.
\end{enumerate}

Indeed, take a point $p\in \mathring M$. From $\bigcup_{j=1}^\infty M_j=M$ 
and {\rm (1$_j$)} we see that there exists a number $j_0\in\n$ such that $p\in M_{j}$ for all $j\geq j_0$. We have
\begin{eqnarray*}
 	\left\| \frac{\partial \wt F}{\theta}(p)  \right\| & \geq &   
 	\left\| \frac{\partial F_{j_0}}{\theta} (p) \right\| - \sum_{j>j_0}   
	\left\| \frac{\partial F_j}{\theta}(p) - \frac{\partial F_{j-1}}{\theta}(p)  \right\|
\\
 	& \stackrel{\text{\rm (2$_j$), (6$_j$)}}{>} &   \left\| \frac{\partial F_{j_0}}{\theta}(p) \right\| - 
	\sum_{j>j_0} \tau_j
\\
 	& \stackrel{\text{\eqref{eq:tauj}}}{\geq} &   \left\| \frac{\partial F_{j_0}}{\theta}(p)  \right\|
	\biggl(1-\sum_{j>j_0}\frac1{2^j}\biggr)
\\
	& > & \frac12   \left\| \frac{\partial F_{j_0}}{\theta}(p)  \right\| \;>\; 0.
\end{eqnarray*}
Since this holds for each point $p\in \mathring M$,  $\wt F|_{\mathring M}$ is an immersion. 
As $\wt F$ is a uniform limit on $M$ of conformal harmonic immersions,
it is a conformal harmonic immersion on $\mathring M$,  and hence $\wt F\in\CMI_*^0(M,\r^n)$. 

Property  {\rm (3$_j$)} says that for every $k\in \n$ and all $j\ge k$ we have
$\dist_{F_j}(p_0,bM_k)>k$. Property {\rm (2$_j$)} ensures that the sequence $F_j$ converges to
$\wt F$ in the $\Cscr^1(M_k)$ topology, and hence in the limit we get that
$\dist_{\wt F}(p_0,bM_k) \ge k$. Since this holds for all $k\in\n$, we see that 
$\wt F|_{\mathring M}$ is complete, thereby proving {\rm (a)}.

Property {\rm (b)} is a trivial consequence of {\rm (4$_j$)}, $j\in\n$.

In order to check properties {\rm (c)} and {\rm (d)}, pick a pair of distinct points $p,q\in M$. 
If $n\in\{3,4\}$,  assume that $\{p,q\}\subset bM$. 
Choose $j_0\in\n$ such that ${\rm d}(p,q)>\frac1{j}$ for all $j\geq j_0$. Given $j>j_0$ we have
\begin{eqnarray*}
	\|F_{j-1}(p)-F_{j-1}(q)\| & \leq & \|F_j(p)-F_{j-1}(p)\|+\|F_j(q)-F_{j-1}(q)\|   
\\
	 & & +\|F_j(p)-F_j(q)\|
\\
	 & \stackrel{\text{\rm (2$_j$), (6$_j$)}}{<} & 2\varsigma_j+\|F_j(p)-F_j(q)\|
\\
 	& \stackrel{\text{\eqref{eq:varsigmaj}}}{<} & \frac1{j^2}\|F_{j-1}(p)-F_{j-1}(q)\| + \|F_j(p)-F_j(q)\|.
\end{eqnarray*}
Therefore,
\[
	\|F_j(p)-F_j(q)\| > \left(1-\frac1{j^2}\right) \|F_{j-1}(p)-F_{j-1}(q)\|, \quad  j>j_0,
\]
and hence
\[
	\|F_{j_0+k}(p)-F_{j_0+k}(q)\| > \|F_{j_0}(p)-F_{j_0}(q)\| 
	\prod_{j=j_0+1}^{j_0+k} \left(1-\frac1{j^2}\right), \quad  k\in\n.
\]
Taking limits in the above inequality as $k$ goes to infinity, we conclude that 
\[
	\|\wt F(p)-\wt F(q)\|\geq \frac12\|F_{j_0}(p)-F_{j_0}(q)\| > 0,
\]
where the latter inequality in ensured by {\rm (5$_{j_0}$)}. 

This completes the proof of Theorem \ref{th:Jordan}. 
\end{proof}

%%%%%%%%%%
%%%%%%%%%%
%%%%%%%%%%
%%%%%%%%%%	SECTION 5
%%%%%%%%%%
%%%%%%%%%%

\section{Complete proper minimal surfaces in convex domains} \label{sec:proper}

In this section we prove a technical result, Theorem \ref{th:proper2}, which will be used
in the following section to prove Theorems \ref{th:proper}, \ref{th:topology} and Corollary \ref{co:bdddomains}.

% Let us first recall some background on convex sets theory and fix the notation.

Assume that $\Dscr$ is a smoothly bounded, strictly convex domain in $\r^n$, $n\geq 3$.
We denote by $\nu_\Dscr$ the inner normal to $b\Dscr$ and by $\kappa_\Dscr^{\rm max}$ and $\kappa_\Dscr^{\rm min}$ the maximum 
and the minimum of the principal curvatures of points in $b\Dscr$ with respect to $\nu_\Dscr$. Obviously, 
$0<\kappa_\Dscr^{\rm min}\leq\kappa_\Dscr^{\rm max}$. For any real number $-\infty<t<1/\kappa_\Dscr^{\rm max}$ we denote by 
$\Dscr_t$ the smoothly bounded, strictly convex domain bounded by $b\Dscr_t=\{p+t\nu_\Dscr(p)\colon p\in b\Dscr\}$. 
Clearly, $t_1<t_2<1/\kappa_\Dscr^{\rm max}$ implies $\Dscr_{t_2}\Subset\Dscr_{t_1}$, and 
\begin{equation}\label{eq:curvature}
	\frac1{\kappa_{\Dscr_t}^{\rm max}}=\frac1{\kappa_{\Dscr}^{\rm max}}-t,\quad \frac1{\kappa_{\Dscr_t}^{\rm min}}
	=\frac1{\kappa_{\Dscr}^{\rm min}}-t,\quad \text{for all }t<\frac1{\kappa_\Dscr^{\rm max}}.
\end{equation} 

By the classical Minkowski theorem, every convex domain in $\r^n$ can be exhausted by an increasing sequence of 
smoothly bounded, strictly convex domains.

%%%%%%%%%%
%%%%%%%%%%
%%%%%%%%%%   PROPER MINIMAL SURFACES IN STRICTLY CONVEX
%%%%%%%%%%
%%%%%%%%%%
%%%%%%%%%%

%
%
%  COMPLETE PROPER CONFORMAL MINIMAL IMMERSIONS 
%
%
\begin{theorem}\label{th:proper2}
Let $n\geq 3$ be a natural number, let $\Lscr\Subset\Dscr\Subset\r^n$ be smoothly bounded (of class at least $\Cscr^2$)
strictly convex domains, and let $\eta>0$ be such that $\Dscr\subset\Lscr_{-\eta}$. 
Let $M$ be a compact bordered Riemann surface and let $F\in \CMI^1(M,\r^n)$ be a conformal minimal immersion with 
$F(bM)\subset \Dscr\setminus\overline\Lscr$. Given a  number $\mu >0$ and a compact set $K\subset \mathring M$ 
there exists a continuous map $\wt F\colon M\to\overline\Dscr$ satisfying the following conditions:
\begin{enumerate}[\rm (i)]
\item $\wt F|_{\mathring M}\colon\mathring M\to \Dscr$ is a complete proper conformal minimal immersion.
\item $\wt F(bM)\subset b\Dscr$ is a finite family of closed curves.
\item $\Flux(\wt F)=\Flux(F)$.
\item $\|\wt F-F\|_{0,M}< \sqrt{2 \eta^2+ {2 \eta}/{\kappa_\Lscr^{\rm min}}}$.
\item $\|\wt F-F\|_{1,K}<\mu$.
\end{enumerate} 
If $n\geq 5$ then  we can choose $\wt F$ to be an embedding on $\mathring M$.
\end{theorem}

Unfortunately we are unable to ensure that the frontier 
$\wt F(b M)\subset b\Dscr$ consists of Jordan curves even when $n\geq 5$;
the reason is explained in Remark \ref{rem:Jordan} at the end of the section.

The proof Theorem \ref{th:proper2}  uses an inductive procedure in which we alternately apply
the following two types of deformations to a conformal minimal immersion $F\colon M\to \r^n$:
\begin{itemize}
\item[\rm (i)] Push the boundary $F(bM)$  closer to  $b\Dscr$
while keeping the resulting immersion suitably close to $F$ in the 
$\Cscr^0(M)$ sense, depending on how far is $F(bM)$ from $b\Dscr$.
This deformation is provided by Lemma \ref{lem:proper} below.
\item[\rm (ii)] Increase the interior boundary distance of the immersion by a prescribed 
(arbitrarily big) amount by a deformation which is arbitrarily small in the $\Cscr^0(M)$ 
sense. Such deformation is provided by Lemma \ref{lem:Jordan2} in Sec.\ \ref{sec:Jordan}.
\end{itemize}
The resulting sequence of conformal minimal immersions $F_k\colon M\to \Dscr$
$(k\in\z_+)$ converges uniformly on $M$ to a continuous map $\wt F\colon M\to\overline\Dscr$
satisfying Theorem \ref{th:proper2}. 

We begin with technical preparations.

%%%%%%%%%%
%%%%%%%%%%
%%%%%%%%%%
%%%%%%%%%%  Lemma proper
%%%%%%%%%%
%%%%%%%%%%

\begin{lemma}\label{lem:proper}
Let $n\ge 3$, let $\Lscr\Subset\Dscr$ be smoothly bounded, strictly convex domains in $\r^n$, 
and let $\eta>0$ be such that $\Dscr\subset \Lscr_{-\eta}$. 
Let $M$ be a compact bordered Riemann surface and let $F\in\CMI_*^1(M,\r^n)$. 
Assume that for a compact set $K\subset \mathring M$ we have
\begin{equation}\label{eq:lemproper}
	F(M\setminus \mathring K)\subset \Dscr\setminus\overline\Lscr.
\end{equation}
Given a number $\delta$   
satisfying $0<\delta<1/\kappa_\Dscr^{\rm max}$,  $F$ can be approximated as closely as desired 
in the $\Cscr^1(K)$ topology by a conformal minimal immersion $\wt F\in\CMI_*^1(M,\r^n)$ 
enjoying the following properties:
\begin{enumerate}[\rm (a)]
\item $\|\wt F-F\|_{0,M}< \sqrt{2 \eta^2+ {2 \eta}/{\kappa_\Lscr^{\rm min}}}$.
\item $\wt F(M\setminus \mathring K)\subset \Dscr\setminus\overline\Lscr$.
\item $\wt F(bM)\subset \Dscr\setminus\overline\Dscr_\delta$.
\item $\Flux(\wt F)=\Flux(F)$. 
\end{enumerate}
\end{lemma}

\begin{proof} 
The main idea is to perturb $F$ near $bM$ in such a way that the image of each point $p\in bM$ is moved into 
the convex shell $\Dscr\setminus\overline\Dscr_\delta$  by pushing it in a direction orthogonal to the inner unit normal $\nu_\Lscr(p)$ of $b\Lscr$ at $p$. 
By Pythagoras' theorem and basic theory of convex domains, condition \eqref{eq:lemproper} ensures that it will be enough to push each point a distance 
smaller than $\sqrt{\eta^2+ {2\eta}/{\kappa_\Lscr^{\rm min}}}$, allowing us to guarantee condition {\rm (a)}.

We may assume that $M$ is a smoothly bounded domain in an open Riemann surface $\wt M$. 
Without loss of generality we may also assume that $\delta>0$ is small enough so that $\overline\Lscr\subset \Dscr_\delta$.

In view of \eqref{eq:lemproper} we may choose a constant $\varsigma>0$ such that 
\begin{equation}\label{eq:eta0}
	\overline\Lscr_{-\varsigma}\subset\Dscr_{\delta},\quad F(M\setminus \mathring K)\subset \Dscr\setminus\overline\Lscr_{-\varsigma}.
\end{equation}
Pick another constant $c>0$ to be specified later. For every point $\bx\in b\Lscr$ set 
\begin{equation} \label{eq:cball}
B_\bx:=b\Lscr \cap \b_{\bx}(c),
\end{equation}
 where $\b_{\bx}(c)$ denotes the open euclidean ball in $\r^n$ centered at $\bx$ with radius $c>0$. Set
\begin{equation}\label{eq:Ox}
	O_\bx=\Dscr\cap\{\by-t\nu_\Lscr(\by) \colon  \by\in B_\bx,\; t>\varsigma\}\subset \Dscr\setminus\overline\Lscr_{-\varsigma}.
\end{equation} 
Assume that $c>0$ is small enough so that $B_\bx$ is a topological open ball and
\begin{equation}\label{eq:Ox1}
	O_\bx \subset \wt O_\bx:= \{\by\in\Dscr \colon \langle\by-\bx,\nu_\Lscr(\bx)\rangle 
 	<-\varsigma/2\} \subset \Dscr\setminus\overline\Lscr_{-\varsigma/2} \quad \forall \bx\in b\Lscr
\end{equation}
\begin{figure}[ht]
    \begin{center}
    \scalebox{0.28}{\includegraphics{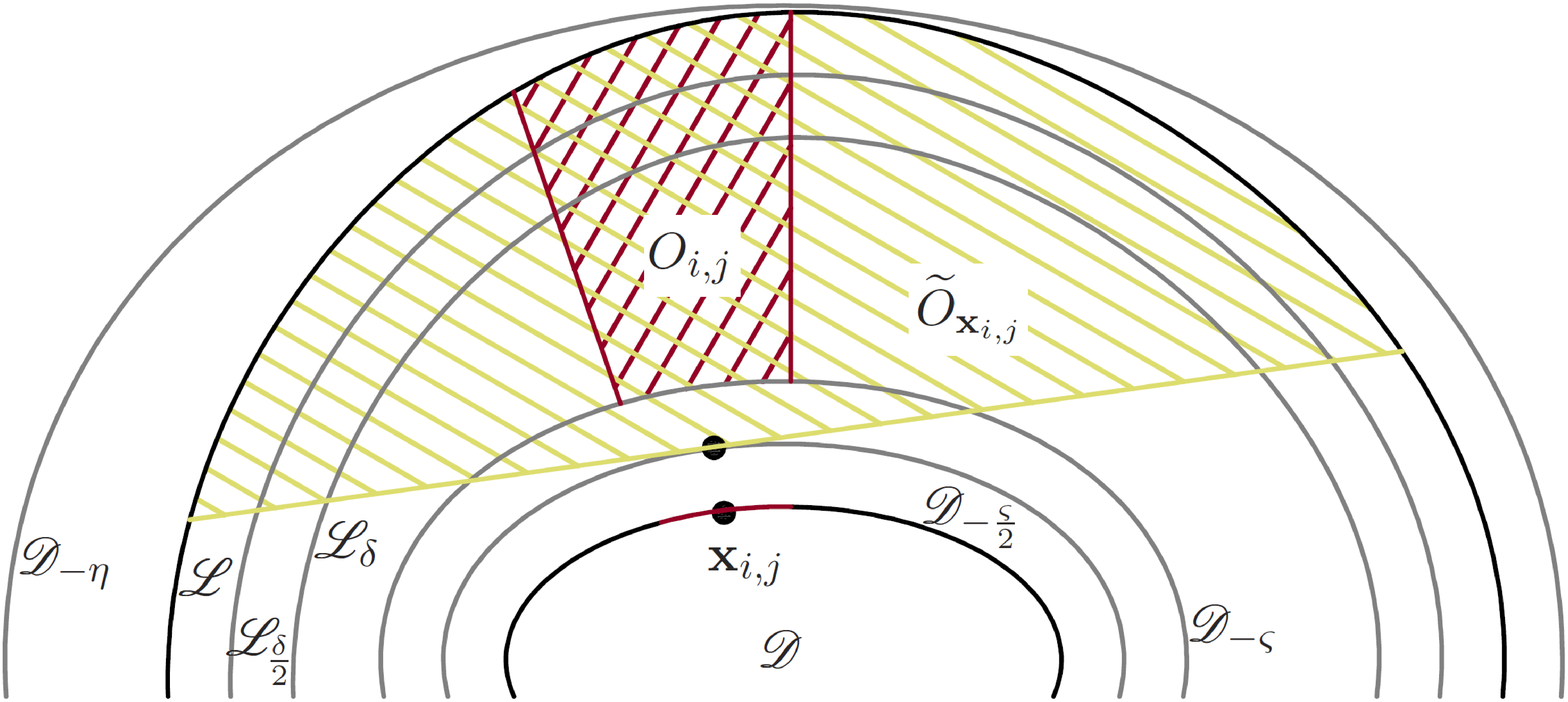}}
        \end{center}
\caption{The convex domains $\Lscr$ and $\Dscr$.}
\label{fig:D-L}
\end{figure}
(see Figure \ref{fig:D-L}). Observe that \eqref{eq:Ox1}  holds in the limit case $c=0$.
Set $\Ocal:=\{O_\bx\colon \bx\in b\Lscr\}$ and notice that $\Dscr\setminus\overline\Lscr_{-\varsigma}=\bigcup_{\bx\in b\Lscr}O_\bx$. 
Denote by $\alpha_1,\ldots,\alpha_k$ the connected boundary curves of $M$. Since $\Ocal$ is an open covering of the compact set 
$F(bM)\subset\Dscr\setminus\overline\Lscr_{-\varsigma}$ (cf. \eqref{eq:eta0}), there exist a natural number $l\geq 3$ and 
compact connected subarcs $\{\alpha_{i,j}\colon (i,j) \in \ttJ:=\{1,\ldots,k\}\times \z_l\}$ satisfying 
\[
	\bigcup_{j\in\z_l} \alpha_{i,j}=\alpha_i\quad \text{for all $i\in\{1,\ldots,k\}$},
\]
and, for every $(i,j)\in\ttJ$, $\alpha_{i,j}$ and $\alpha_{i,j+1}$ have a common endpoint $p_{i,j}$ and are otherwise disjoint, 
$\alpha_{i,j}\cap\alpha_{i,a}=\emptyset$ for all $a\in\z_l\setminus\{j-1,j,j+1\}$, and
\begin{equation}\label{eq:alphaU}
	F(\alpha_{i,j})\subset O_{i,j}:=O_{\bx_{i,j}}\in\Ocal\quad \text{for some $\bx_{i,j}\in b\Lscr$}.
\end{equation}

\begin{lemma}\label{lem:arcs}
(Notation as in Lemma \ref{lem:proper}.) Let $\varsigma>0$ be
such that (\ref{eq:eta0}) holds.
Given a number $\epsilon_1>0$, there exists $F_0\in\CMI_*^1(M,\r^n)$ satisfying the following properties:
\begin{enumerate}[\rm ({P}1)]
\item $F_0$ is $\epsilon_1$-close to $F$ in the $\Cscr^1(K)$ topology.
\item $F_0(p_{i,j})\in \Dscr\setminus\overline\Dscr_{\delta/2}$ for all $(i,j)\in\ttJ$.
\item $F_0(\alpha_{i,j})\subset O_{i,j}$ for all $(i,j)\in\ttJ$.
\item $F_0(M\setminus \mathring K)\subset \Dscr\setminus \overline\Lscr_{-\varsigma}$.
\item $\Flux(F_0)=\Flux (F)$.
\end{enumerate}
\end{lemma}
\begin{proof}
For each $(i,j)\in\ttJ$ we choose an arc $\gamma_{i,j}\subset \wt M$ with the endpoint $p_{i,j}\in bM$ 
and otherwise disjoint from $M$ such that the arcs $\gamma_{i,j}$, $(i,j)\in\ttJ$, are pairwise disjoint and 
\[
	S:=M\cup\bigcup_{(i,j)\in\ttJ} \gamma_{i,j}\subset\wt M
\]
is an {\em admissible set} in the sense of \cite[Def.\ 5.1]{AFL}. 
(This implies that the Mergelyan approximation theorem holds on $S$.)
Let $v\colon S\to\Dscr\subset \r^n$ be a smooth map satisfying the following conditions:
\begin{enumerate}[\rm (i)]
\item $v=F$ on a neighborhood of $M$.
\item $v(\gamma_{i,j})\subset O_{i,j}\cap O_{i,j+1}$ and $v(q_{i,j})\in \Dscr\setminus \overline\Dscr_{\delta/2}$, where $q_{i,j}$ 
is the endpoint of $\gamma_{i,j}$ different from $p_{i,j}$, for all $(i,j)\in\ttJ$. 
(Observe that 
$O_{i,j}\cap O_{i,j+1}=\Dscr\cap\{\by-t\nu_\Lscr(\by) \colon  \by\in B_{\bx_{i,j}}\cap B_{\bx_{i,j+1}},\; t>\varsigma\}\neq\emptyset$. 
See \eqref{eq:Ox}, \eqref{eq:alphaU}, and Figure \ref{fig:D-L}).
\end{enumerate}

Pick a number $\epsilon_2>0$ which will be specified later. As in the proof of Lemma \ref{lem:step1}, we may use the Mergelyan theorem 
for conformal minimal immersions \cite[Theorem 5.3]{AFL} to obtain $G\in\CMI_*(\wt M,\r^n)$ satisfying the following properties:
\begin{enumerate}[\rm (i)]
\item[\rm (iii)]  $G$ is $\epsilon_2$-close to $v$ in the $\Cscr^1(M)$ and the $\Cscr^0(S)$ topologies.
\item[\rm (iv)] $\Flux(G)=\Flux(F)$.
\end{enumerate}

Let $V\subset \wt  M$ be a small open neighborhood of $S$. If $\epsilon_2>0$ is small enough, 
then properties {\rm (i)}, {\rm (ii)}, and {\rm (iii)} ensure the existence of small open neighborhoods 
$U_{i,j}$ of $\alpha_{i,j}$, $V_{i,j}$ of $\gamma_{i,j}$, and $W'_{i,j}\Subset W_{i,j}$ of $p_{i,j}$ in 
$V \setminus K$, $(i,j)\in \ttJ$,  satisfying the following conditions:
\begin{enumerate}[\rm (i)]
\item[\rm (v)] $V_{i,j}\cap M\Subset W'_{i,j}\Subset W_{i,j}\Subset U_{i,j}\cap U_{i,j+1}\Subset V\setminus K$.
\item[\rm (vi)] $G(V_{i,j-1}\cup U_{i,j}\cup V_{i,j})\subset O_{i,j}$. (Take into account \eqref{eq:alphaU}.)
\item[\rm (vii)] $G(q_{i,j})\in \Dscr\setminus\Dscr_{\delta/2}$.
\end{enumerate}
Without loss of generality we assume in addition that the compact sets $\overline W_{i,j}\cup\overline V_{i,j}$, $(i,j)\in \ttJ$, 
are pairwise disjoint. 

By \cite[Theorem 2.3]{FW0} (see also \cite[Theorem 8.8.1, p.\ 365]{F2011}), there exists a smooth 
diffeomorphism $\phi\colon M\to\phi (M)\subset V$ satisfying the following properties:
\begin{enumerate}[\rm (i)]
\item[\rm (viii)] $\phi \colon  \mathring M \to  \phi(\mathring M)$ is a biholomorphism.
\item[\rm (ix)] $\phi$ is as close as desired to the identity in the $\Cscr^1$ topology on $M \setminus \bigcup_{(i,j)\in \ttJ} W'_{i,j}$. 
\item[\rm (x)] $\phi(p_{i,j}) = q_{i,j}\in b\,\phi(M)$ and $\phi(M \cap W'_{i,j}) \subset W_{i,j}\cup V_{i,j}$ for all $(i,j)\in \ttJ$.
\end{enumerate}

Set $F_0:=G\circ\phi\in\CMI_*^1(M,\r^n)$.
If $\epsilon_2>0$ is small enough and the approximation in {\rm (ix)} is close enough, then $F_0$ satisfies the conclusion of the lemma. 
Indeed, {\rm (P1)} is ensured by {\rm (i)}, {\rm (iii)}, and {\rm (ix)}; %F do we wish to add (v)?
{\rm (P2)} by {\rm (vii)} and {\rm (x)}; {\rm (P3)} by {\rm (ix)}, {\rm (x)}, {\rm (v)}, and {\rm (vi)}; {\rm (P4)} by {\rm (ix)}, 
{\rm (i)}, {\rm (iii)}, \eqref{eq:eta0}, {\rm (x)}, {\rm (vi)}, and \eqref{eq:Ox}; and {\rm (P5)} by {\rm (iv)} and the definition of $F_0$.
\end{proof}

We continue with the proof of Lemma \ref{lem:proper}. 
Let $F_0\in\CMI_*^1(M,\r^n)$ be provided by Lemma \ref{lem:arcs} for some $\epsilon_1>0$ which will be specified later.
In view of {\rm (P2)} and {\rm (P3)}, each arc $\alpha_{i,j}$ contains a proper connected compact subarc 
$I_{i,j}\Subset\alpha_{i,j}$ such that 
\begin{equation}\label{eq:IO}
	F_0(\overline{\alpha_{i,j}\setminus I_{i,j}})\subset (\Dscr\setminus\overline\Dscr_{\delta/2})\cap O_{i,j}.
\end{equation}
Here $O_{i,j}:=O_{\bx_{i,j}}\in\Ocal$ for a certain point $\bx_{i,j}\in b\Lscr$, cf.\ \eqref{eq:alphaU}.

Pick an annular neighborhood $A\subset M\setminus K$ of $bM$ and a smooth retraction $\rho\colon A\to bM$.
Choose pairwise disjoint, smoothly bounded closed disc $\overline D_{i,j}$ in $A$, $(i,j)\in\ttJ$, such that 
\begin{equation}\label{eq:Dij2}
	I_{i,j}\Subset \overline D_{i,j}\cap\alpha_{i,j}, \quad 
	\rho(\overline D_{i,j})\subset \overline D_{i,j}\cap\alpha_{i,j}, 
	\quad\text{and}\quad F_0(\overline D_{i,j})\subset O_{i,j}.
\end{equation}

Set $I:=\bigcup_{(i,j)\in\ttJ} I_{i,j}$. Choose pairs of unitary orthogonal vectors 
$\{\bu_{i,j},\bv_{i,j}\}\subset\langle \nu_\Lscr(\bx_{i,j})\rangle^\bot$, where $\bx_{i,j}\in b\Lscr$ were given in 
\eqref{eq:alphaU}, $(i,j)\in \ttJ$, and consider a continuous map  
$\varkappa\colon bM \times\overline{\d}\to\r^n$ of the form
\[
	\varkappa(x,\xi)=\left\{\begin{array}{ll}
	F_0(x),   & x\in bM\setminus I \\
	F_0(x) + r(x)\,(\Re \sigma(x,\xi) \bu_{i,j} + \Im \sigma(x,\xi)  \bv_{i,j}), & x\in I_{i,j},\; (i,j)\in\ttJ,
	\end{array}\right.
\] 
where $r\colon bM\to\r_+$ and $\sigma\colon I\times\overline\d\to\c$ are functions as in Theorem \ref{th:RH} such that 
\begin{equation}\label{eq:convexdiscs}
  	\varkappa(bM\times \t)\subset \Dscr\setminus\overline\Dscr_\delta.
\end{equation}
Such functions clearly exist; one can for instance take $r$ and $\sigma$ so that $\varkappa(x,\overline\d)$ is the planar 
convex disc $\overline\Dscr_{\delta/2}\cap (F_0(x)+\span\{\bu_{i,j},\bv_{i,j}\})$ for all $x\in bM$.

From properties  {\rm (P3)}, \eqref{eq:Ox1}, \eqref{eq:IO}, and \eqref{eq:convexdiscs}
we infer that
\begin{equation}\label{eq:varDT}
	\varkappa(bM\times \overline\d)\subset \Dscr\setminus\overline\Lscr_{-\varsigma/2}.
\end{equation}

Fix a number $\epsilon_3>0$ which will be specified later.  Theorem \ref{th:RHCMI} furnishes arbitrarily small open neighborhoods 
$\Omega_{i,j}\subset \overline D_{i,j}$ of   $I_{i,j}$ in $M$, $(i,j)\in \ttJ$, and a conformal minimal immersion 
$\wt F\in \CMI_*^1(M,\r^n)$ satisfying the following properties:
\begin{enumerate}[\rm ({P}1)]
\item[\rm (P6)] $\dist(\wt F(x),\varkappa(x,\t))<\epsilon_3$ for all $x\in bM$.
\item[\rm (P7)] $\dist(\wt F(x),\varkappa(\rho(x),\overline{\d}))<\epsilon_3$ for all $x\in \Omega:=\bigcup_{(i,j)\in\ttJ}\Omega_{i,j}\subset M\setminus K$.
\item[\rm (P8)] $\wt F$ is $\epsilon_3$-close to $F_0$ in the $\Cscr^1$ topology on $M\setminus \Omega$.
\item[\rm (P9)] $\Flux (\wt F)=\Flux (F_0)$.
\end{enumerate}

Let us check that the immersion $\wt F$ satisfies the conclusion of Lemma \ref{lem:proper} provided that
the positive numbers $\epsilon_1$ and $\epsilon_3$ are small enough.

First of all, properties {\rm (P8)} and {\rm (P1)} ensure that $\wt F$ is as close to $F$ in the $\Cscr^1(K)$ topology as desired
if $\epsilon_1$ and $\epsilon_3$ are small enough (observe that $K\subset M\setminus \Omega$).

Pick a point $p\in bM$ and let $(i,j)\in\ttJ$ with $p\in\alpha_{i,j}$. 
In view of {\rm (P3)}, \eqref{eq:varDT}, \eqref{eq:Ox1}, and {\rm (P6)}, we have 
\begin{equation}\label{eq:wtFp}
\wt F(p)\in \wt O_{\bx_{i,j}} \setminus\overline\Dscr_\delta \subset \Dscr.
\end{equation}
This proves condition {\rm (c)} in the lemma. Since $\Lscr_{-\eta}$ is smoothly bounded and strictly convex, it is contained in the 
Euclidean ball in $\r^n$ centered at $\bx_{i,j}+\frac1{\kappa_\Lscr^{\rm min}}\nu_\Lscr(\bx_{i,j})$ with radius 
$1/\kappa^{\rm min}_{\Lscr_{-\eta}}=\eta + 1/{\kappa_\Lscr^{\rm min}}$; cf.\ \eqref{eq:curvature}. Therefore, 
taking into account that $\Dscr\subset \Lscr_{-\eta}$ and \eqref{eq:wtFp}, Pythagoras' theorem ensures that 
\[
 \|\wt F(p)- (\bx_{i,j}-t\nu_\Lscr(\bx_{i,j}))\|<\sqrt{2\eta^2+\big(\frac2{\kappa_\Lscr^{\rm min}} 
  -\varsigma\big)\eta-\frac{\varsigma}{\kappa_\Lscr^{\rm min}}} \quad \text{for all $t\in \big(\frac{\varsigma}2,\eta\big)$.}
\]
Since $F(p)$ lies in the convex domain $\wt O_{\bx_{i,j}}\subset\Lscr_{-\eta}$ (see \eqref{eq:alphaU} and \eqref{eq:Ox1}), 
we have that $t_p:=\langle F(p)-\bx_{i,j}, -\nu_\Lscr(\bx_{i,j})\rangle \in (\varsigma/2,\eta)$.
Together with \eqref{eq:alphaU} and \eqref{eq:Ox}, and taking into account \eqref{eq:curvature}, basic trigonometry gives
\[
\|F(p)-(\bx_{i,j}-t_p\nu_\Lscr(\bx_{i,j}))\|<c (\eta \kappa_{\Lscr}^{\rm max}+1) \sqrt{1-c^2 (\kappa_{\Lscr}^{\rm max})^2/4},
\] 
where $c>0$ is the constant given in \eqref{eq:cball}. 
The last two inequalities ensure $\|\wt F(p)-F(p)\|< \sqrt{2\eta^2+2\eta/\kappa_\Lscr^{\rm min}}$, proving {\rm (a)}, 
provided that $c>0$ is chosen small enough.

In order to check {\rm (b)}, notice that, if $\epsilon_3>0$ is sufficiently small, {\rm (P8)} and {\rm (P4)} give that 
$\wt F(M\setminus (K\cup\Omega))\subset \Dscr\setminus \overline\Lscr_{-\varsigma}$. On the other hand, {\rm (P7)} 
and \eqref{eq:varDT} guarantee that $\wt F(\Omega)\subset \Dscr\setminus \overline\Lscr_{-\varsigma/2}$ as well.

Finally, {\rm (P5)} and {\rm (P9)} trivially imply {\rm (d)}. This concludes the proof of  Lemma \ref{lem:proper}.
\end{proof}

%%%%%%%%%%
%%%%%%%%%%
%%%%%%%%%%
%%%%%%%%%%  PROOF OF THEOREM PROPER2
%%%%%%%%%%
%%%%%%%%%%

\begin{proof}[Proof of Theorem \ref{th:proper2}]
By the Mergelyan theorem for conformal minimal immersions \cite{AFL}, we may assume that $F\in\CMI_*^1(M,\r^n)$. 
Moreover, if $n\geq 5$, we may also assume that $F$ is an embedding; see Theorem \ref{th:gp}. 

Let  $\mu >0$ and $\eta>0$ be as in the theorem. Since $F(bM)\subset \Dscr\setminus\overline\Lscr$, there exist 
 a number $\epsilon >0$ and a smoothly bounded compact domain $K_0\subset \mathring M$ which is a 
strong deformation retract of $M$ such that $K\subset K_0$,  $\overline \Lscr_{-\epsilon}\subset \Dscr$, and
\[
F(M\setminus \mathring K_0)\subset \Dscr\setminus\overline\Lscr_{-\epsilon}.
\]
Since $\epsilon>0$, it follows from \eqref{eq:curvature} that 
$\sqrt{2(\eta-\epsilon)^2+{2(\eta-\epsilon)}/{\kappa_{\Lscr_{-\epsilon}}^{\rm min}}}< \sqrt{2\eta^2+{2\eta}/{\kappa_\Lscr^{\rm min}}}$; 
we therefore may choose a sequence $-1/\kappa_\Dscr^{\rm max}>\delta_1>\delta_2>\cdots >\lim_{j\to\infty}\delta_j=0$ satisfying 
$\overline\Lscr_{-\epsilon}\subset \Dscr_{\delta_1}$ and 
\begin{equation}\label{eq:sumetas}
\sqrt{2(\eta-\epsilon)^2+\frac{2(\eta-\epsilon)}{\kappa_{\Lscr_{-\epsilon}}^{\rm min}}}+
\sum_{j\geq 1} \sqrt{2\delta_j^2+\frac{2\delta_j}{\kappa_\Dscr^{\rm min}}} < 
 \sqrt{2\eta^2+\frac{2\eta}{\kappa_\Lscr^{\rm min}}}.
\end{equation}
Set $F_0:=F$, $\delta_0:=\eta-\epsilon$, $\Bscr^0:=\Lscr_{-\epsilon}$, and $\Bscr^j:=\Dscr_{\delta_j}$ for all $j\in\n$. 
Fix a point $p_0\in \mathring K$ and a number $\epsilon_0$ with $0<\epsilon_0<\mu$.

By recursively applying Lemma \ref{lem:proper}, Lemma \ref{lem:Jordan2}, and Theorem \ref{th:gp} we may 
construct a sequence $\{\Xi_j=(K_j, F_j, \epsilon_j)\}_{j\in\n}$, where 
$K_j$ is a smoothly bounded compact domain in $\mathring M$ which is a strong deformation retract of $M$
and we have $\bigcup_{j\in\n} K_j=\mathring M$,  $F_j\in\CMI_*^1(M,\r^n)$, and $\epsilon_j>0$, 
satisfying the following conditions for all $j\in\n$:
\begin{enumerate}[\rm (a{$_j$})]
\item $K_{j-1}\subset \mathring K_j$.
\item $F_j$ is $\epsilon_j$-close to $F_{j-1}$ in the $\Cscr^1(K_{j-1})$ topology.
\item $\|F_j-F_{j-1}\|_{0,M}<\sqrt{2\delta_{j-1}^2+{2\delta_{j-1}}/{\kappa_{\Bscr^{j-1}}^{\rm min}}}$.
\item $F_j(M\setminus \mathring K_{j-1})\subset \Dscr\setminus \overline{\Bscr^{j-1}}$.
\item $F_j(M\setminus \mathring K_j)\subset \Dscr\setminus \overline{\Bscr^j}$.
\item $\Flux(F_j)=\Flux (F)$.
\item $\dist_{F_j}(p_0, bK_i)>i$ for all $i\in\{0,\ldots,j\}$.
\item If $n\geq 5$, then $F_j$ is an embedding.
\item If $n\ge 5$ then $\epsilon_j<\min\left\{ \epsilon_{j-1}/2, \tau_j,\varsigma_j\right\}$, 
where the number $\tau_j$ is defined by (\ref{eq:tauj}) and 
\[ 
	\varsigma_j= \frac1{2j^2} \inf\left\{ \|F_{j-1}(p)-F_{j-1}(q)\| : p,q\in M,\ {\rm d}(p,q)>\frac1{j}\right\}.
\] 
If $n=3,4$ then $\epsilon_j<\min\left\{ \epsilon_{j-1}/2, \tau_j\right\}$.
\end{enumerate}

Notice that $\Xi_0=(F_0,K_0,\epsilon_0)$ satisfies {\rm (e$_0$)},  {\rm (f$_0$)}, {\rm (g$_0$)}, and {\rm (h$_0$)},  
whereas the other conditions  are void for $j=0$. Let $j\in\n$ and assume  the existence of triples $\Xi_0,\ldots, \Xi_{j-1}$ 
enjoying these  conditions. Fix $\epsilon_j>0$ such that  (${\rm i}_j$) holds. Applying Lemma \ref{lem:proper}
to the data
\[
\Lscr=\Bscr^{j-1},\quad \Dscr,\quad \eta=\delta_{j-1},\quad M,\quad F=F_{j-1},\quad K=K_{j-1},\quad \delta=\delta_j,
\]
we obtain $F_j\in\CMI_*^1(M,\r^n)$ satisfying {\rm (b$_j$)}, {\rm (c$_j$)}, {\rm (d$_j$)}, {\rm (f$_j$)}, 
$F_j(bM)\subset \Dscr\setminus\overline{\Bscr^j}$, and $\dist_{F_j}(p_0, bK_i)>i$ for all $i\in\{0,\ldots,j-1\}$. 
Therefore we may choose $K_j\Subset M$ fulfilling conditions {\rm (a$_j$)} and {\rm (e$_j$)}. 

We now apply Lemma \ref{lem:Jordan2} to approximate $F_j$  uniformly on $M$ 
by a conformal minimal immersion $\wh F_j\in \CMI_*^1(M,\r^n)$ such that $\wh F_j(M) \subset \Dscr$,
$\Flux(\wh F_j)=\Flux (F)$,  and $\dist_{\wh F_j}(p_0,bM) > j$. 
Assuming as we may that the approximation is close enough, 
$\wh F_j$ satisfies all the properties of $F_j$ that we have verified so far.
Replacing $F_j$ by $\wh F_j$ and enlarging the set $K_j$ if necessary we may  assume that 
condition {\rm (g$_j$)} holds as well.  

Finally, Theorem \ref{th:gp} enables us to ensure condition {\rm (h$_j$)}, thereby closing the induction.

Properties $({\rm c}_j)$ and \eqref{eq:sumetas} guarantee that the sequence $\{F_j\}_{j\in\n}$ converges in the 
$\Cscr^0(M)$ topology  to a continuous map $\wt F= \lim_{j\to+\infty} F_j\colon M\to\r^n$ satisfying 
Theorem \ref{th:proper2} {\rm (iv)}; take into account that $\Bscr^j=\Dscr_{\delta_j}$ and so 
$\kappa_{\Bscr^j}^{\rm min}>\kappa_{\Dscr}^{\rm min}$ for all $j\in\n$ (cf.\ \eqref{eq:curvature}).
From conditions {\rm (b$_j$)} and {\rm (i$_j$)} we obtain that
\begin{equation}\label{eq:properness}
	\text{$\wt F$ is $\epsilon_j$-close to $F_j$ in the $\Cscr^1(K_j)$ topology for all $j\in\z_+$.}
\end{equation}
In particular, $\wt F$ is $\epsilon_0$-close to $F_0=F$ in the $\Cscr^1(K)$ topology;
since $\epsilon_0<\mu$, property {\rm (v)} in Theorem \ref{th:proper2} holds.
Furthermore, as in the proof of Theorem \ref{th:Jordan}, we see that conditions %{\rm (b$_j$)}, 
{\rm (g$_j$)},  {\rm (h$_j$)}, {\rm (i$_j$)}, and \eqref{eq:properness} ensure that $\wt F|_{\mathring M}\colon \mathring M\to\r^n$ 
is a complete minimal immersion which is an embedding if $n\geq 5$. % , so properties {\rm (i)} and {\rm (ii)}  hold. 
Finally, property {\rm (f$_j$)} give that $\Flux(\wt F)=\Flux(F)$, whereas {\rm (d$_j$)} and \eqref{eq:properness} guarantee 
that $\wt F(\mathring M)\subset\Dscr$ and $\wt F|_{\mathring M}\colon \mathring M\to \Dscr$ is a proper map; recall that 
$\Bscr^j=\Dscr_{\delta_j}$ for all $j\in\n$ and that $\{\delta_j\}_{j\in\n}\searrow 0$.
Since $\wt F$ is continuous on $M$, it follows that $\wt F(bM)\subset b\Dscr$ is a finite
family of curves.
\end{proof}

%
%   REMARK ON JORDAN BOUNDARIES
%

\begin{remark}\label{rem:Jordan}
Our method does not ensure that the map $\wt F|_{bM}\colon bM\to\r^n$ in Theorem \ref{th:proper2} is an embedding. 
The reason is that, at each step in the recursive procedure, we can only assert that $F_j$ is 
$\sqrt{2\delta_{j-1}^2+{2\delta_{j-1}}/{\kappa_{\Dscr}^{\rm min}}}$-close to $F_{j-1}$ in the 
$\Cscr^0(bM)$ topology (cf.\ {\rm (c$_j$)}), and the number $\delta_{j-1}$ is given a priori in the construction of 
$F_{j-1}$ (or in other words, $F_{j-1}$ depends on $\delta_{j-1}$). To guarantee embeddedness of  $\wt F(bM)$ a more 
accurate approximation, depending on the geometry of $F_{j-1}(bM)$, would be required.
\qed\end{remark}

%%%%%%%%%%
%%%%%%%%%%
%%%%%%%%%%   SECTION 6
%%%%%%%%%%
%%%%%%%%%%
%%%%%%%%%%

\section{Proof of Theorems \ref{th:proper}, \ref{th:topology} and Corollary \ref{co:bdddomains}} \label{sec:topology}

%
%  PROOF OF THEOREM 1.2
%

\subsection{Proof of Theorem \ref{th:proper}}
Part (a) is a direct consequence of Theorem \ref{th:proper2}; 
indeed for $F$ and $\Dscr$ as in Theorem \ref{th:proper} {\rm (a)}, just take any smoothly bounded, 
strictly convex domain $\Lscr\Subset \Dscr$ with $F(bM)\cap \overline\Lscr=\emptyset$ and apply Theorem \ref{th:proper2}. 

We now prove part (b). 
Let $F\colon M\to\overline \Dscr$ be as in Theorem \ref{th:proper} {\rm (b)} and let $\epsilon>0$. Up to a translation we may 
assume without loss of generality that the origin $0\in\r^n$ lies in $\Dscr$. Set $d:=\max\{\|x\|\colon x\in b\Dscr\}>0$. 
Fix $\lambda\in(0,\min\{1,1/2d\kappa_{\Dscr}^{\rm max}\})$ to be specified later. Set $F_0:=(1-\lambda) F\in\CMI^1(M,\r^n)$ 
and observe that  $F_0(bM)\subset \Dscr\setminus\overline \Dscr_{2\lambda d}$ and 
\begin{equation}\label{eq:co1}
\|F_0-F\|_{0,M}\leq \lambda d.
\end{equation}
Theorem \ref{th:proper2} applied to the data $\Lscr=\Dscr_{2\lambda d},$ $\Dscr$, $\eta=2\lambda d$, and $F=F_0$, 
furnishes a continuous map $\wt F\colon M\to \overline\Dscr$ such that $\wt F(\mathring M)\subset \Dscr$, 
$\wt F|_{\mathring M}\colon \mathring M\to\Dscr$ is a conformal complete proper minimal immersion 
(embedding if $n\geq 5$), $\Flux(\wt F)=\Flux(F_0)$, and 
\[
	\big\|\wt F-F_0\big\|_{0,M}<\sqrt{8\lambda^2 d^2+{4\lambda d}\left(\frac1{\kappa_\Dscr^{\rm min}}-2\lambda d\right)}
\]
 (take into account \eqref{eq:curvature}). Together with \eqref{eq:co1} we obtain that $\|\wt F-F\|_{0,M}<\epsilon$
 provided that $\lambda>0$ is chosen sufficiently small. This shows that the flux can be changed by an arbitrarily small amount 
 when passing from $F$ to $\wt F$. This completes the proof of Theorem \ref{th:proper}.

%
%  PROOF OF THEOREM 1.4
%

\subsection{Proof of Theorem \ref{th:topology}}

Let $D\subset \r^n$ $(n\geq 3)$ be a convex domain. Take an exhaustion 
$\Bscr^0\Subset\Bscr^1\Subset\cdots \Subset \cup_{j\in\z_+} \Bscr^j=D$ of $D$ by smoothly bounded, strictly convex domains 
$\Bscr^j\subset\r^n$. Choose a sequence $\{\lambda_j\}_{j\in\z_+}\searrow 0$ with $0<\lambda_j<1/\kappa_{\Bscr^j}^{\rm max}$,  
and denote by $\delta_j$ the Hausdorff distance between 
$\overline{\Bscr^j_{\lambda_j}}$ and $\overline{\Bscr^{j+1}}$  for all $j\in\z_+$. It follows that $\delta_j>\lambda_j$ and 
$\Bscr^{j+1}\subset\Bscr^j_{-\delta_j+\lambda_j}$ for all $j\in\z_+$. (Observe that possibly $\Bscr^j_{-\delta_j+\lambda_j}\nsubseteq D$.)

%
%   PROOF OF PART (a)
%

\medskip
\noindent{\em Proof of part (a)}. Let $M$ be a compact bordered Riemann surface.  Let $F_0\in\CMI_*^1(M,\r^n)$ be an immersion
 (embedding if $n\geq 5$) satisfying $F_0(bM)\subset \Bscr^1\setminus\overline{\Bscr^0}$. Choose $K_0$ any smoothly bounded 
 compact domain in $\mathring M$ which is a strong deformation retract of $M$ and with 
 $F_0(M\setminus \mathring K_0)\subset \Bscr^0\setminus\overline{\Bscr^0_{\lambda_0}}$, and any number $\epsilon_0>0$. 
 As in the proof of Theorem \ref{th:proper2}, we may recursively apply Lemma \ref{lem:proper}, Lemma \ref{lem:Jordan2}, 
 and Theorem \ref{th:gp} in order to construct a sequence $\{\Xi_j=(K_j, F_j, \epsilon_j)\}_{j\in\n}$, where 
$K_j$ is a smoothly bounded compact domain in $\mathring M$ which is a strong deformation retract of $M$ and we have 
$\bigcup_{j\in\n} K_j=\mathring M$,  $F_j\in\CMI_*^1(M,\r^n)$, and $\epsilon_j>0$, 
satisfying the following conditions for all $j\in\n$:
\begin{enumerate}[\rm (a{$_j$})]
\item $K_{j-1}\subset \mathring K_j$.
\item $\|F_j-F_{j-1}\|_{1,K_{j-1}}<\epsilon_j$.
\item $\|F_j-F_{j-1}\|_{0,M}<\sqrt{2\delta_{j-1}^2+{2\delta_{j-1}}/{\kappa_{\Bscr_{\lambda_{j-1}}^{j-1}}^{\rm min}}}$.
\item $F_j(M\setminus \mathring K_{j-1})\subset \Bscr^j\setminus \overline{\Bscr_{\lambda_{j-1}}^{j-1}}$. 
\item $F_j(M\setminus \mathring K_j)\subset \Bscr^j\setminus \overline{\Bscr_{\lambda_j}^j}$.
\item $\Flux(F_j)=\Flux (F_0)$.
\item $\dist_{F_j}(p_0, bK_i)>i$ for all $i\in\{0,\ldots,j\}$.
\item If $n\geq 5$, then $F_j$ is an embedding.
\item If $n\ge 5$ then $\epsilon_j<\min\left\{ \epsilon_{j-1}/2, \tau_j,\varsigma_j\right\}$, 
where the number $\tau_j$ is defined by (\ref{eq:tauj}) and 
\[ 
	\varsigma_j= \frac1{2j^2} \inf\left\{ \|F_{j-1}(p)-F_{j-1}(q)\| : p,q\in M,\ {\rm d}(p,q)>\frac1{j}\right\}.
\] 
If $n=3,4$ then $\epsilon_j<\min\left\{ \epsilon_{j-1}/2, \tau_j\right\}$.
\end{enumerate}
(Property {\rm (c$_j$)} is useless in this proof and can be ruled out. In fact, unlikely in Theorem \ref{th:proper2}, it does not enable us 
to ensure that the sequence $\{F_j\}_{j\in\z_+}$ converges up to $bM$;  see Remark \ref{rem:continuous} for a more detailed explanation.) 
In this case, to pass from $F_{j-1}$ to $F_j$ in the inductive step we apply Lemma \ref{lem:proper} to the data
\[
\Lscr=\Bscr_{\lambda_{j-1}}^{j-1},\quad \Dscr=\Bscr^j,\quad \eta=\delta_{j-1},\quad F=F_{j-1},\quad \delta=\lambda_j.
\]
As in the proof of Theorem \ref{th:proper2}, and taking into account that $D=\cup_{j\in\z_+}\Bscr^j$, these properties ensure that 
$\{F_j\}_{j\in\n}$ converges uniformly on compact subsets of $\mathring M$ to a conformal complete proper minimal immersion
 (embedding if $n\geq 5$) $\wt F\colon \mathring M\to D$. Furthermore, since $F_0$ is full then $\wt F$ is also full provided the 
 $\epsilon_j$'s are chosen small enough at each step. This concludes the proof of part {\rm (a)}.

%
%   PROOF OF PART (b)
%
\medskip
\noindent{\em Proof of part (b)}. 
Let $\wt M$ be an open Riemann surface and let $\pgot\colon H_1(\wt M;\z)$ be a group homomorphism. 
Exhaust $\wt M$ by an increasing sequence $M_0\subset M_1\subset\cdots\subset \bigcup_{j=0}^\infty M_j=\wt M$ of 
compact smoothly bounded connected Runge regions  such that $M_0$ is a disc and the Euler characteristic of 
$M_j\setminus\mathring M_{j-1}$ satisfies $\chi(M_j\setminus\mathring M_{j-1})\in\{0,-1\}$ for all $j\in\n$.

Set $K_0:=M_0$ and let $F_0\in\CMI_*^1(K_0,\r^n)$ be an immersion (embedding if $n\geq 5$) satisfying 
$F_0(K_0)\subset \Bscr^0\setminus\overline{\Bscr^0_{\lambda_0}}$. Fix $\epsilon_0>0$ and a point $p_0\in \mathring M$. 
We shall construct a sequence $\{\Xi_j=(K_j, F_j, \epsilon_j)\}_{j\in\z_+}$ where $K_j\subset M_j$ is a smoothly bounded 
compact Runge domain which is a strong deformation retract of $M_j$, $F_j\in\CMI_*^1(K_j,\r^n)$, and $\epsilon_j>0$, 
satisfying the following conditions:
\begin{enumerate}[\rm (a{$_j$})]
\item[\rm (a{$_j$})] $K_{j-1}\subset \mathring K_j$.
\item[\rm (b{$_j$})] $\|F_j-F_{j-1}\|_{1,K_{j-1}}<\epsilon_j$.
%
%\item[\rm (c{$_j'$})] The Hausdorff distance ${\rm dist^H}(F_{j-1}(K_{j-1}),F_j(K_j))<\epsilon_j+\sqrt{2\delta_{j-1}^2+{2\delta_{j-1}}/{\kappa_{\Bscr_{\lambda_{j-1}}^{j-1}}^{\rm min}}}$.
%
\item[\rm (d{$_j$})] $F_j(K_j\setminus \mathring K_{j-1})\subset \Bscr^j\setminus \overline{\Bscr_{\lambda_{j-1}}^{j-1}}$.
\item[\rm (e{$_j$})] $F_j(b K_j)\subset \Bscr^j\setminus \overline{\Bscr_{\lambda_j}^j}$.
\item[\rm (f{$_j$})] $\Flux(F_j)=\pgot|_{H_1(K_j;\z)}$.
\item[\rm (g{$_j$})] $\dist_{F_j}(p_0, bK_i)>i$ for all $i\in\{0,\ldots,j\}$.
\item[\rm (h{$_j$})] If $n\geq 5$, then $F_j$ is an embedding.
\item[\rm (i{$_j$})] If $n\ge 5$ then $\epsilon_j<\min\left\{ \epsilon_{j-1}/2, \tau_j,\varsigma_j\right\}$, 
where the number $\tau_j$ is defined by (\ref{eq:tauj}) and 
\[ 
	\varsigma_j= \frac1{2j^2} \inf\left\{ \|F_{j-1}(p)-F_{j-1}(q)\| : p,q\in K_{j-1},\ {\rm d}(p,q)>\frac1{j}\right\}.
\] 
If $n=3,4$ then $\epsilon_j<\min\left\{ \epsilon_{j-1}/2, \tau_j\right\}$.
\end{enumerate}
(Observe that there is no property {\rm (c{$_j$})} in the above list; this is not a misprint, we labeled the properties in this way in 
order to emphasize that, under our current assumptions, a condition similar to {\rm (c{$_j$})} in the proof of part {\rm (a)} is not expected.)
The triple $\Xi_0=(K_0,F_0,\epsilon_0)$ meets the above conditions for $j=0$, except for {\rm (a$_0$)}, {\rm (b$_0$)}, {\rm (d$_0$)}, 
and {\rm (i$_0$)} which are void. For the inductive step, assume that we have triples $\Xi_0,\ldots,\Xi_{j-1}$ satisfying the required
properties for some $j\in\n$ and let us construct $\Xi_j$. Fix $\epsilon_j>0$ to be specified later. Let us distinguish cases depending
on whether the Euler characteristic $\chi(M_j\setminus\mathring M_{j-1})$ equals $0$ or $-1$.

\smallskip
\noindent{\em Case 1:} $\chi(M_j\setminus\mathring M_{j-1})=0$. In this case there
is no change of topology when passing from $M_{j-1}$ to $M_j$. Therefore, $K_{j-1}$ is a strong deformation retract of $M_j$. 
By the Mergelyan Theorem for conformal minimal immersions \cite[Theorem 5.3]{AFL} 
we may find a smoothly bounded compact region $K_j\subset M_j$ and may approximate $F_{j-1}$ by a map
$\wt F_j\in\CMI_*^1(K_j,\r^n)$ such that the triple $\wt\Xi_j=(K_j,\wt F_j,\epsilon_j)$ satisfies {\rm (a$_j$)},  {\rm (b$_j$)}, 
%
%${\rm dist^H}(F_{j-1}(K_{j-1}),\wt F_j(K_j))<\epsilon_j$, 
%
{\rm (d$_j$)}, {\rm (f$_j$)}, and {\rm (g$_j$)} for $i=0,\ldots,j-1$. (For {\rm (d$_j$)} take into account {\rm (e$_{j-1}$)}.) 
Applying Lemma \ref{lem:proper} to the data
\[
M=K_j,\quad \Lscr=\Bscr_{\lambda_{j-1}}^{j-1},\quad \Dscr=\Bscr^j,\quad \eta=\delta_j,\quad F=\wt F_j,\quad \delta=\lambda_j,
\]
we obtain $F_j\in\CMI_*^1(K_j,\r^n)$ such that the triple $\Xi_j=(K_j,F_j,\epsilon_j)$ meets condition {\rm (e$_j$)} in addition to 
the above properties. Finally, by Lemma \ref{lem:Jordan2} and Theorem \ref{th:gp}, we may also assume that {\rm (g$_j$)}, 
{\rm (h$_j$)}, and {\rm (i$_j$)} are also satisfied.

\smallskip
\noindent{\em Case 2:} $\chi(M_j\setminus\mathring M_{j-1})=-1$. In this case there exists a smooth arc 
$\gamma\subset \mathring M_j\setminus \mathring K_{j-1}$ with both endpoints in $bK_{j-1}$ and otherwise disjoint with 
$K_{j-1}$ such that $\chi(M_j\setminus(\mathring K_{j-1}\cup\gamma))=0$. We may also assume that 
$S=K_{j-1}\cup\gamma\Subset M_j$ is an admissible subset in the sense of \cite[Def.\ 5.1]{AFL}. 
Extend $F_{j-1}$ to a generalized conformal minimal immersion $(F_{j-1},f\theta)$ on $S$, 
in the sense of \cite[Def.\ 5.2]{AFL}, such that 
\begin{equation}\label{eq:tildeF}
\text{$F_{j-1}(\gamma)\subset \Bscr^{j-1}\setminus \overline{\Bscr_{\lambda_{j-1}}^{j-1}}$\; and\; 
$\int_\alpha \Im(f\theta)=\pgot(\alpha)$ for every closed curve $\alpha\subset S$;}
\end{equation}
take into account {\rm (e$_{j-1}$)} and {\rm (f$_{j-1}$)}. 
By the Mergelyan theorem for conformal minimal immersions \cite[Theorem 5.3]{AFL} we may approximate $F_{j-1}$ on $S$ by maps $\wt F_{j-1}\in\CMI_*^1(M_j,\r^n)$. 
If the approximation is close enough then, taking into account Theorem \ref{th:gp}, 
there exists  a smoothly bounded compact Runge region $L_{j-1}\Subset M_j$ which is is a strong deformation retract of $M_j$ 
and satisfies $S\Subset L_{j-1}$ and 
$\chi(M_j\setminus \mathring{L}_{j-1})=0$, such that the triple $\wt\Xi_{j-1}=(L_{j-1},\wt F_{j-1},\epsilon_{j})$ satisfies
\begin{enumerate}[\rm (a{$_{j}$})]
\item[\rm ($\wt{\rm a}${$_{j}$})] $K_{j-1}\Subset \mathring{L}_{j-1}\subset M_j$.
%
%\item[\rm ($\wt{\rm b}${$_{j}$})] $L_{j-1}$ is a strong deformation retract of $M_j$.
%
\item[\rm ($\wt{\rm c}${$_{j}$})] $\|\wt F_{j-1}-F_{j-1}\|_{1, K_{j-1}}<\epsilon_j/2$.
\item[\rm ($\wt{\rm d}${$_{j}$})]  $\wt F_{j-1}(L_{j-1}\setminus \mathring K_{j-1})\subset \Bscr^{j-1}\setminus \overline{\Bscr_{\lambda_{j-1}}^{j-1}}$. 
(Take into account {\rm (e$_{j-1}$)} and \eqref{eq:tildeF}.)
\item[\rm ($\wt{\rm f}${$_{j}$})]  $\Flux(\wt F_{j-1})=\pgot|_{H_1(L_{j-1};\z)}$. 
(Take into account {\rm (f$_{j-1}$)} and \eqref{eq:tildeF}.)
\item[\rm ($\wt{\rm g}${$_{j}$})]  $\dist_{\wt F_{j-1}}(p_0, bK_i)>i$ for all $i\in\{0,\ldots,j-1\}$. 
(Take into account {\rm (g$_{j-1}$)}.) 
In particular, $\dist_{\wt F_{j-1}}(p_0, bL_{j-1})>j-1$; see {\rm ($\wt{\rm a}$$_j$)}.
\item[\rm ($\wt{\rm h}${$_{j}$})]  If $n\geq 5$ then $\wt F_{j-1}$ is an embedding.
\end{enumerate}
This reduces the proof to Case 1, closing the induction and concluding the construction of the sequence $\{\Xi_j\}_{j\in\z_+}$.

Set $\Rcal:=\bigcup_{j\in\z_+} K_j$. Since $\bigcup_{j\in\z_+} M_j=\wt M$ and $K_j$ is a strong deformation retract of $M_j$ 
for all $j\in\z_+$, property {\rm (a$_j$)} ensures that $\Rcal\subset \wt M$ is an open domain homeomorphic to $\wt M$. 
Given $\epsilon>0$, properties {\rm (b$_j$)} and {\rm (i$_j$)} guarantee that we may choose the numbers $\epsilon_j>0$ small enough 
in the inductive construction so that the sequence $\{F_j\}_{j\in\z_+}$ converges uniformly on compact subsets of $\Rcal$ to a 
conformal minimal immersion $\wt F:=\lim_{j\to\infty} F_j\colon \Rcal\to \r^n$ which is $\epsilon$-close to $F$ in the 
$\Cscr^1(M)$ topology; recall that $K_0=M$. Further, if the $\epsilon_j$'s are chosen sufficiently small, conditions 
{\rm (d$_j$)}, {\rm (f$_j$)}, {\rm (g$_j$)}, and {\rm (h$_j$)} guarantee that $\wt F(\Rcal)\subset D$, $\wt F\colon\Rcal\to D$ 
is a proper map, $\Flux(\wt F)=\pgot$, $\wt F$ is complete, and, if $n\geq 5$, $\wt F$ is an embedding. 
(Recall that $\cup_{j\in\z_+}\Bscr^j=D$.) See the proof of Theorem \ref{th:proper2} for details on how to check these properties. 
This concludes the proof of part {\rm (b)}.

%
%   REMARK - CONTINUOUS
%

\begin{remark}\label{rem:continuous}
Our method does not ensure that $\wt F\colon \mathring M\to D$ in Theorem \ref{th:topology} {\rm (a)} 
extends continuously up to $bM$. The reason is that, at each step in the recursive process, we only have that $F_j$ is 
$\sqrt{2\delta_{j-1}^2+{2\delta_{j-1}}/{\kappa_{\Bscr_{\lambda_{j-1}}^{j-1}}^{\rm min}}}$-close to $F_{j-1}$ 
in the $\Cscr^0(bM)$ topology (see {\rm (c$_j$)}). Since the domains $\Bscr^j$'s need not be parallel to each other, 
this sequence is not necessarily Cauchy (in fact neither 
$\{\delta_j\}_{j\in\z_+}$ nor $\{1/\kappa_{\Bscr_{\lambda_j}^{j-1}}^{\rm min}\}_{j\in\z_+}$ need to 
be bounded sequences in general) and so we do not get convergence of the sequence $\{F_j\}_{j\in\z_+}$ up to $bM$.
\qed \end{remark}

%
%
%   PROOF OF COROLLARY 1.5
%
%
\subsection{Proof of Corollary \ref{co:bdddomains}}

Let $D$ be a domain in $\r^n$ with a smooth strictly convex boundary point $x\in bD$, that is to say, $bD$ is smooth and 
has positive principal curvatures with respect to the inner normal in a neighborhood of $x$.
There exist a number $r>0$ and a smoothly bounded, strictly convex domain $\Dscr\subset D$ such that $x\in b\Dscr$ and 
$U:=b\Dscr\cap \b(x,r)\subset bD$, where $\b(x,r)\subset\r^n$ denotes the Euclidean ball centered at $x$ with radius $r$. 
Fix a number $\lambda\in(0,1/\kappa_{\Dscr}^{\rm max})$ to be specified later. 

\medskip \noindent {\em Proof of part  (a).}
Let $M$ be a compact bordered Riemann surface and let $F\in\CMI^1_*(M,\r^n)$ be a conformal minimal immersion satisfying 
$F(M)\subset \Dscr\setminus\overline\Dscr_\lambda$ and 
\begin{equation}\label{eq:co2}
	\|F(p)-x\|<\lambda\quad \text{for all $p\in M$.}
\end{equation}
Theorem \ref{th:proper2}, applied to the data $\Lscr=\Dscr_\lambda,$ $\Dscr$, $\eta=\lambda$, and $F$, furnishes a continuous map 
$\wt F\colon M\to \overline\Dscr$ such that $\wt F(\mathring M)\subset \Dscr\subset D$, 
$\wt F|_{\mathring M}\colon \mathring M\to\Dscr$ is a conformal complete proper minimal immersion (embedding if $n\geq 5$), 
and 
\[
	\big\|\wt F-F\big\|_{0,M}<\sqrt{2\lambda^2 +{2\lambda}\left(\frac1{\kappa_\Dscr^{\rm min}}-\lambda\right)}
\]
(take into account \eqref{eq:curvature}). In view of \eqref{eq:co2} we get that $\|\wt F(p)-x\|<r$ for all $p\in M$, provided that 
$\lambda>0$ is chosen small enough. Since $\wt F(bM)\subset b\Dscr$, we obtain that $\wt F(bM)\subset U\subset bD$ and hence 
$\wt F|_{\mathring M}\colon \mathring M\to D$ is proper. Finally, since $F$ is full, $\wt F$ is also full provided the 
approximation is close enough.  This completes the proof of part (a).

%
%   PROOF OF PART (b)
%
\medskip \noindent  {\em Proof of part  (b).}
Pick a number $r'$ with $0<r'<r$ and a decreasing sequence 
$\{\lambda_j\}_{j\in\z_+}\searrow 0$ with $0<\lambda_j<\min\{r',1/2\kappa_{\Dscr}^{\rm max}\}$ for all $j\in\z_+$. 
These constants will be specified later. Set $\Bscr^j=\Dscr_{\lambda_j}$ and $\delta_j=\lambda_j-\lambda_{j+1}$ for all $j\in\z_+$.
Let $\wt M$, $\pgot$, and $\{M_j\}_{j\in\z_+}$ be as in the proof of Theorem \ref{th:topology} {\rm (b)}. 
Let $F_0\colon M_0\to \r^n$ be a nondegenerate conformal minimal immersion with 
$F_0(M_0)\subset \b(x,r')\cap(\Bscr^0\setminus\overline{\Bscr^0_{\lambda_0}})$.  
As in the proof of Theorem \ref{th:topology} {\rm (b)} we may recursively construct a sequence 
$\{\Xi_j=(K_j,F_j,\epsilon_j)\}_{j\in\z_+}$ satisfying conditions {\rm (a$_j$)}, {\rm (b$_j$)}, {\rm (d$_j$)},{\rm (e$_j$)}, 
{\rm (f$_j$)}, {\rm (g$_j$)}, {\rm (h$_j$)}, and {\rm (i$_j$)} there, and also
\begin{enumerate}[\rm (a$_j$)]
\item[\rm (c$_j$)] $F_j(K_j)\subset \b\Big(x,r'+\sum_{i=1}^j\sqrt{2\delta_{i-1}^2+{2\delta_{i-1}}/{\kappa_{\Bscr^{i-1}_{\lambda_{i-1}}}^{\rm min}}}\Big)$ for all $j\in\z_+$.
\end{enumerate}
Indeed, this extra condition is directly granted by Lemma \ref{lem:proper} {\rm (a)} when passing from $\Xi_{j-1}$ to $\Xi_j$; 
in case $\chi(M_j\setminus \mathring M_{j-1})=-1$ we take the arc $\gamma$ so that 
$F_{j-1}(\gamma)\subset (\Bscr^{j-1}\setminus\overline{\Bscr^{j-1}_{\lambda_{j-1}}}) \cap \b\Big(x,r'+\sum_{i=1}^{j-1}\sqrt{2\delta_{i-1}^2+{2\delta_{i-1}}/{\kappa_{\Bscr^{i-1}_{\lambda_{i-1}}}^{\rm min}}}\Big)$, 
which is possible in view of {\rm (c$_{j-1}$)} and {\rm (e$_{j-1}$)}.

Taking into account that $\delta_j=\lambda_j-\lambda_{j+1}$ and 
$1/\kappa_{\Bscr^j_{\lambda_j}}^{\rm min}=-2\lambda_j+1/\kappa_{\Dscr}^{\rm min}$ (cf.\ \eqref{eq:curvature}), 
the above properties ensure that the sequence $\{F_j\}_{j\in\z_+}$ converges to a conformal complete proper nondegenerate 
minimal immersion (embedding if $n\geq 5$) $\wt F\colon \Rcal\to\Dscr$, where $\Rcal=\cup_{j\in\z_+} M_j$ is homeomorphic to 
$\wt M$, satisfying $\Flux(\wt F)=\pgot$ and $\wt F(\Rcal)\subset b\Dscr\cap\b(x,r)\subset bD$, provided that $r'$ and the $\lambda_j$'s 
are chosen small enough. This concludes the proof.

%%%%%%%%%%
%%%%%%%%%%
%%%%%%%%%%
%%%%%%%%%%  ACKNOWLEDGEMENTS
%%%%%%%%%%
%%%%%%%%%%

\subsection*{Acknowledgements}
A. Alarc\'on is supported by the Ram\'on y Cajal program of the Spanish Ministry of Economy and Competitiveness.

A.\ Alarc\'{o}n and F.\ J.\ L\'opez are partially supported by MCYT-FEDER grant MTM2011-22547 and 
Junta de Andaluc\'ia grant P09-FQM-5088. 

B.\ Drinovec Drnov\v sek and F.\ Forstneri\v c are partially  supported  by the research program P1-0291 and the grant J1-5432 
from ARRS, Republic of Slovenia. 

Part of this work was made when F.\ Forstneri\v c visited the institute IEMath-Granada with support by the GENIL-SSV 2014 program.

Part of this work was made when B.\ Drinovec Drnov\v sek visited University of Oslo. She would like 
to thank the Department of Mathematics for the hospitality and partial financial support.

%%%%%%%%%%
%%%%%%%%%%
%%%%%%%%%%
%%%%%%%%%%   THE BIBLIOGRAPHY
%%%%%%%%%%
%%%%%%%%%%

% \vspace{2mm}

\vskip 13mm

\noindent Antonio Alarc\'{o}n

\noindent Departamento de Geometr\'{\i}a y Topolog\'{\i}a, Universidad de Granada, E--18071 Granada, Spain.

\noindent  e-mail: {\tt alarcon@ugr.es}

\vspace*{0.3cm}
\noindent Barbara Drinovec Drnov\v sek

\noindent Faculty of Mathematics and Physics, University of Ljubljana, and Institute
of Mathematics, Physics and Mechanics, Jadranska 19, SI--1000 Ljubljana, Slovenia.

\noindent e-mail: {\tt barbara.drinovec@fmf.uni-lj.si}

\vspace*{0.3cm}

\noindent Franc Forstneri\v c

\noindent Faculty of Mathematics and Physics, University of Ljubljana, and Institute
of Mathematics, Physics and Mechanics, Jadranska 19, SI--1000 Ljubljana, Slovenia.

\noindent e-mail: {\tt franc.forstneric@fmf.uni-lj.si}

\vspace*{0.3cm}

\noindent Francisco J.\ L\'opez

\noindent Departamento de Geometr\'{\i}a y Topolog\'{\i}a, Universidad de Granada, E--18071 Granada, Spain.

\noindent  e-mail: {\tt fjlopez@ugr.es}
\end{document}